\documentclass[11pt]{article}
\usepackage{amsmath,amssymb,amsthm,graphicx,float,url}
 \usepackage[colorlinks=true,citecolor=blue,filecolor=red,urlcolor=magenta]{hyperref}
\usepackage{pdfsync}
\usepackage[usenames, dvipsnames]{xcolor}
\usepackage{stmaryrd}
\usepackage{dsfont}
 \usepackage{bm}
 \usepackage[Algorithme]{algorithm}
 \usepackage{algorithmic}

\numberwithin{equation}{section}
\numberwithin{figure}{section}

\newcommand{\FF}{\mathcal{F}}
\newcommand{\I}{{\mathcal{I}}}
\newcommand{\R}{\textrm{I\kern-0.21emR}}
\newcommand{\N}{\textrm{I\kern-0.21emN}}

\newcommand{\T}{\mathbb{T}}

\newcommand{\B}{\mathcal{B}}

\newcommand{\U}{\mathcal{U}_L}
\newcommand{\V}{\mathcal{V}_L}

\newcommand{\Hn}{\mathcal{H}^{d-1}}

\newcommand{\W}{\Omega}
\renewcommand{\geq}{\geqslant}
\renewcommand{\leq}{\leqslant}
\def\1e{\mathds{1}}

\newtheorem{theorem}{Theorem}[section]

\newtheorem{rem}{Remark}[section]

\theoremstyle{definition}\newtheorem{remark}{Remark}[section]

\title{How locating sensors in thermo-acoustic tomography?\footnote{
Y. Privat was partially supported by the Project ``Analysis and simulation of optimal shapes - application to lifesciences'' of the Paris City Hall.}}

\author{
Ma\"itine Bergounioux\footnote{Institut Denis Poisson, Universit\'e d'Orl\'eans, CNRS UMR 7013, 45067 Orl\'eans, France ({\tt maitine.bergounioux@univ-orleans.fr}).}
 \and \'Elie Bretin\footnote{Univ Lyon, INSA de Lyon, CNRS UMR 5208, Institut Camille Jordan, 20 avenue Albert Einstein, F-69621 Villeurbanne Cedex, France
({\tt elie.bretin@insa-lyon.fr}) }
	\and Yannick Privat\footnote{IRMA, Universit\'e de Strasbourg, CNRS UMR 7501, 7 rue Ren\'e Descartes, 67084 Strasbourg, France ({\tt yannick.privat@unistra.fr}).}
}

\date{\today}

\begin{document}

\maketitle

\begin{abstract}
Thermo-acoustic tomography is a non-invasive medical imaging technique, constituting a precise and cheap alternative to X-imaging. The principle is to excite a body to reconstruct with a pulse inducing an inhomogeneous heating and therefore expansion of tissues. This creates an acoustic wave pressure which is measured with sensors. The reconstruction of heterogeneities inside the body can be then performed by solving an inverse problem, knowing measurements of the acoustic waves outside the body. As the intensity of the measured pressure is expected to be small, a challenging problem consists in locating the sensors in a adequate way.

This paper is devoted to the determination of an optimal sensors location to achieve this reconstruction. 
We first introduce a model involving a least square functional standing for an observation of the pressure for a first series of measures by sensors, and an observability-like constant functional describing for the quality of reconstruction. 
Then, we determine an \textit{appropriate }location of sensors for two series of measures, in two steps: first, we reconstruct possible initial data by solving a worst-case design like problem. Second, we determine from the knowledge of these initial conditions the optimal location of sensors for observing in the best way the corresponding solution of the wave equation.

Far from providing an intrinsic solution to the general issue of locating sensors, solving this problem allows to determine a new sensors location improving the quality of reconstruction before getting a new series of measures. 
 
 We perform a mathematical analysis of this model:
 in particular we investigate existence issues and introduce a numerical algorithm to solve it. Eventually, several numerical 2D simulations illustrate our approach.
\end{abstract}


\medskip



\section{Introduction}\label{sec:intro}

We are interested in thermo-acoustic tomography which is a non-invasive medical imaging technique, constituting a precise and cheap alternative to X-imaging. The principle of this imaging process is quite simple: the tissue to be imaged is irradiated by a pulse and this energy induces an heating process. If the pulse is a radio-frequency electromagnetic pulse the technique is called thermo-acoustic tomography (TAT); if the pulse is a laser (the frequency is much higher) this is called photo-acoustic tomography (PAT). In any case, this creates a thermally induced pressure jump that propagates as a sound wave, which can be detected via sensors that are located outside the body to image. By detecting the pressure waves, heterogeneities can be observed: this gives important informations as for example the location and/or size of tumors in breast cancer. For more details on the process and the related works we refer to \cite{TATMaxwell,BBHP,Bonnefond, sch2} and the references therein. Roughly speaking, TAT and PAT are two hybrid techniques using electromagnetic waves as an excitation (input) and acoustic waves as an observation (output). Both techniques lead to similar ill-posed inverse problems. The modeling of the direct problem does not lead to the same equations, since physical assumptions are different. However, the process can be described by two equations (or equation systems) : 
\begin{itemize}
\item The first system of equations describes the generation of the heating process inside the body. In the PAT, this system involves the fluence equation (the fluence rate is the average of the light intensity in all directions) which is a diffusion equation \cite{BBHP}; in the TAT case, this equation is replaced by Maxwell equations \cite{TATMaxwell}. Then the temperature is driven by the classical heat equation, than can be neglected in the PAT case because the  high speed of light implies that the thermal effect is \textit{quasi instantaneous}. In both cases, the resulted term is a pressure wave source $p_0$ at time $t=0$.
\item The second equation models the behavior of the acoustic wave once the source $p_0$ is known. It is given by 
\begin{equation} \left\lbrace \begin{array}{ll}
\partial_{tt} p (t,x) - \operatorname{div} (c(x) \nabla p (t,x))=0 & \text{in }(0,T)\times \B,\\
p(0,\cdot) = p_0, &\text{on } \B, \\
\partial_t p(0,\cdot) =0 & \text{on } \B, \\
p=0 & \text{on } (0,T)\times \partial\B,
\end{array}\right. \end{equation}
where $T>0$ is arbitrary and $\B$ is a given ball whose radius will be chosen adequately in the sequel.
Most of the first papers on the subject focused on this second equation to deal with the inversion. Indeed, if we perform measures that allow to recover $p_0$, then one can recover heterogeneities via the quantitative estimates of appropriate physical parameters: these are the diffusion and absorption coefficients in PAT, or the electric susceptibility and the conductivity in TAT.
\end{itemize}
 
 In this paper, we are interested in the following issue:
{\bf from the knowledge of a first series of measures, how to locate sensors before performing a second one, in a relevant way?}
 \\
Such a question is related to the general problem of locating optimally sensors. Therefore, we focus on the pressure equation to get the reconstruction of $p_0$ as best as possible. 
 The forthcoming analysis can be used both in PAT and TAT contexts as well. In the sequel we always refer to TAT process for the sake of readability.

\bigskip

Our goal is hence to determine the sensors location yielding to the best possible reconstruction. Therefore, we do not focus on the inverse problem solving where sensors are fixed but rather on their optimal location. There is a great number of papers on TAT/PAT reconstruction. We refer to the bibliography in \cite{TATMaxwell,BBHP,Bonnefond, sch2} and to the books by Ammari and al. \cite{ammari2008introduction,ammari2004reconstruction}.

To our knowledge, the question of optimal location of sensors in the framework of photoacoustic tomography has not been yet addressed. Note however that the issue of looking for optimal location and/or shapes of sensors for the wave/heat equation in a bounded domain has been in particular investigated in \cite{PTZCont1,PTZObs1,PTZcomplexity,PTZparab,PTZobsND}, with the assumption that in some sense, a reflection phenomenon of the solutions on the boundary of the domain occurs. Such an hypothesis would not be relevant for the application considered here and this is why we propose an alternative approach.

\medskip

In what follows, we will model the issue of determining optimal location or shape of sensors in two ways:
\begin{itemize}
\item[(i)] {\bf First attempt: prescribing the total surface of sensors.} We  first choose to deal with sensors described by the characteristic function of a measurable subset whose Lebesgue measure is prescribed. A possible drawback of such an approach rests upon the fact that optimal solution may not exist or have a large/infinite number of connected components (see Problem \eqref{ODP2}). 
\item[(ii)] {\bf Second attempt: prescribing the maximal number of sensors.} To avoid the emergence of too complex solutions and to make the problem modeling more realistic, we will investigate a second issue, where a maximal number $N_0$ of sensors (more precisely of connected components of the sensor set) is prescribed, in addition to the aforementioned total surface constraint (see Problem \eqref{ODP3}).
\end{itemize}

The paper is organized as follows: next section is devoted to the problem modeling. We justify there the adopted point of view. Next, we perform the mathematical analysis of this optimization problem, including existence results and optimality conditions as well. We end with preliminary numerical results in last section.

\section{Modeling the problem}

\subsection{The PDE (direct) model and the sensors set}\label{sec:prelim}

Throughout this paper, we will use the notation $\1e_A$ to denote the characteristic function of a set $A\subset \R^d$ ($d\ge 1$) which is the function equal to 1 on $A$ and 0 elsewhere. 

As explained in Section \ref{sec:intro}, the acoustic wave is assumed to solve the partial differential equation
\begin{equation}\label{WaveEqn} \left\lbrace \begin{array}{ll}
\partial_{tt} p (t,x) - \operatorname{div} (c(x) \nabla p (t,x))=0 & \text{in }(0,T)\times \B,\\
p(0,\cdot) = p_0, &\text{on } \B, \\
\partial_t p(0,\cdot) =0 & \text{on } \B, \\
p=0 & \text{on } (0,T)\times \partial\B,
\end{array}\right. \end{equation}
where $T>0$ is arbitrary and $\B\subset \R^d$ is a given ball whose radius is chosen in the sequel.

Let us also introduce $\Omega$ as a convex open subset of $\R^d$, representing the body we want to image, which means in particular that
\begin{equation}\label{suppOm}
\operatorname{supp}(p_0)\subset \Omega ,
\end{equation}
where $\operatorname{supp}(p_0)$ denotes the support of $p_0$. The convexity assumption on $\Omega$ is technical, and will be used later to make the set of admissible designs well defined. This will be commented in the sequel. 

For this reason, we will assume that $\B$ is large enough so that $\Omega\subset\B$; the set $\B\setminus\Omega$ stands for the ambient media (water or air), where the wave propagates.

The function $c\in L^\infty(\R^d)$ represents the sound speed and satisfies
$$ c(x) \ge c_0 >0 \quad \text{a.e. } x\in \R^d.$$
A typical choice for $c$ is to assume it is piecewise constant, namely:
\begin{equation}\label{def:c} 
 c=c_1 \1e_{\Omega}+c_2 \1e_{\B \setminus \Omega},\quad \text{with } (c_1,c_2)\in(\R_+^*)^2 ,
 \end{equation}
 with $c_1,~c_2>0$.
 Recall that for every initial $p_0\in H^1_0(\Omega)$, there exists a unique solution $p$ to System \eqref{WaveEqn} and that solution satisfies $p\in \mathcal{C}^1\left([0,T], L^2(\B) \right) \cap \mathcal{C}^0 \left([0,T], H^1_0(\B)\right)$ (see e.g. \cite{DL,Evans}).

\begin{remark}
For (future) numerical purposes, the set $\B$ is chosen bounded but it may be noticed that, if the radius $R$ of $\B$ is large enough, then $p$ vanishes on $\partial\B$ all along the recording process (i.e. for $t\leq T$). Precisely, there exists $R>0$ large enough such that, for every $p_0\in H_0^1(\Omega)$, the solution of problem \eqref{WaveEqn} coincides with the solution of 
\begin{equation}\label{WaveEqn2} \left\lbrace \begin{array}{ll}
\partial_{tt} p(t,x) - \operatorname{div} (c(x) \nabla p(t,x))=0 & \text{in } (0,T)\times \R^d,\\
p(0,\cdot) = p_0(\cdot) &\text{on } \R^d,\\
\partial_t p(0,\cdot) =0 & \text{on } \R^d ,\\
\end{array}\right. \end{equation}
with the convention that the initial datum $p_0$ has been extended by 0 to the whole space $\R^d$. One refers for instance to \cite{Evans}.
\end{remark}

As mentioned in Section \ref{sec:intro}, the inverse problem of recovering $p_0\in H^1_0(\Omega) $ from measurements of $p$ on a set 
$(0,T) \times \Sigma$ (where $\Sigma$ is the sensors location) has been extensively studied the two past decades. We refer to 
\cite{ammari2008introduction,ammari2004reconstruction,BBHP,Bonnefond, sch2} and the references therein for further details. However, the point of view we adopt here is a variational one. Precisely, we do not use any reconstruction formula (exact of not) 
to deal with the inverse problem. Following \cite{TATMaxwell,BBHP,Bonnefond} philosophy, we rather include the unknown parameter in a cost functional in a \textit{least-squared} sense. So far, both the reconstruction of $p_0$ and the sensors location question will be addressed in the same functional. 

\paragraph{The sensors set.} To model the optimal design problem of locating sensors in the best way, we first describe the class of admissible designs/sensors that we consider.

Let us endow the Lipschitz set $\partial\Omega$ with the usual $(d-1)$-dimensional Hausdorff measure $\Hn$. In the sequel, we say by convention that $\Gamma\subset\partial\Omega$ is measurable whenever it is measurable for the Hausdorff measure $\Hn$.

Introduce $\Sigma\subset \R^d$ as the subdomain of $\R^d$ occupied by sensors. Roughly speaking, we will assume that each connected component of $\Sigma$ is located around the boundary of $\Omega$ and has a positive {\em thickness} $\varepsilon$. More precisely, we assume the existence of a measurable set $\Gamma\subset\partial\Omega$ such that
\begin{equation}\label{DefSigma}
\Sigma = \left\lbrace s+\mu\,\nu (s), s\in\Gamma, \mu\in[0,\varepsilon] \right\rbrace ,
\end{equation}
where $ \nu (s)$ denotes the outward unit normal to $\Omega$ at $s$ (see Figure \ref{fig1}).
The set $\Sigma$ is thus supported by the annular ring
$$\widetilde{\partial\Omega} := \lbrace s+\mu\nu (s), s\in \partial\Omega, \mu\in [0,\varepsilon] \rbrace.$$
\begin{figure}[H]
\begin{center}
 \includegraphics[width=0.8 \linewidth]{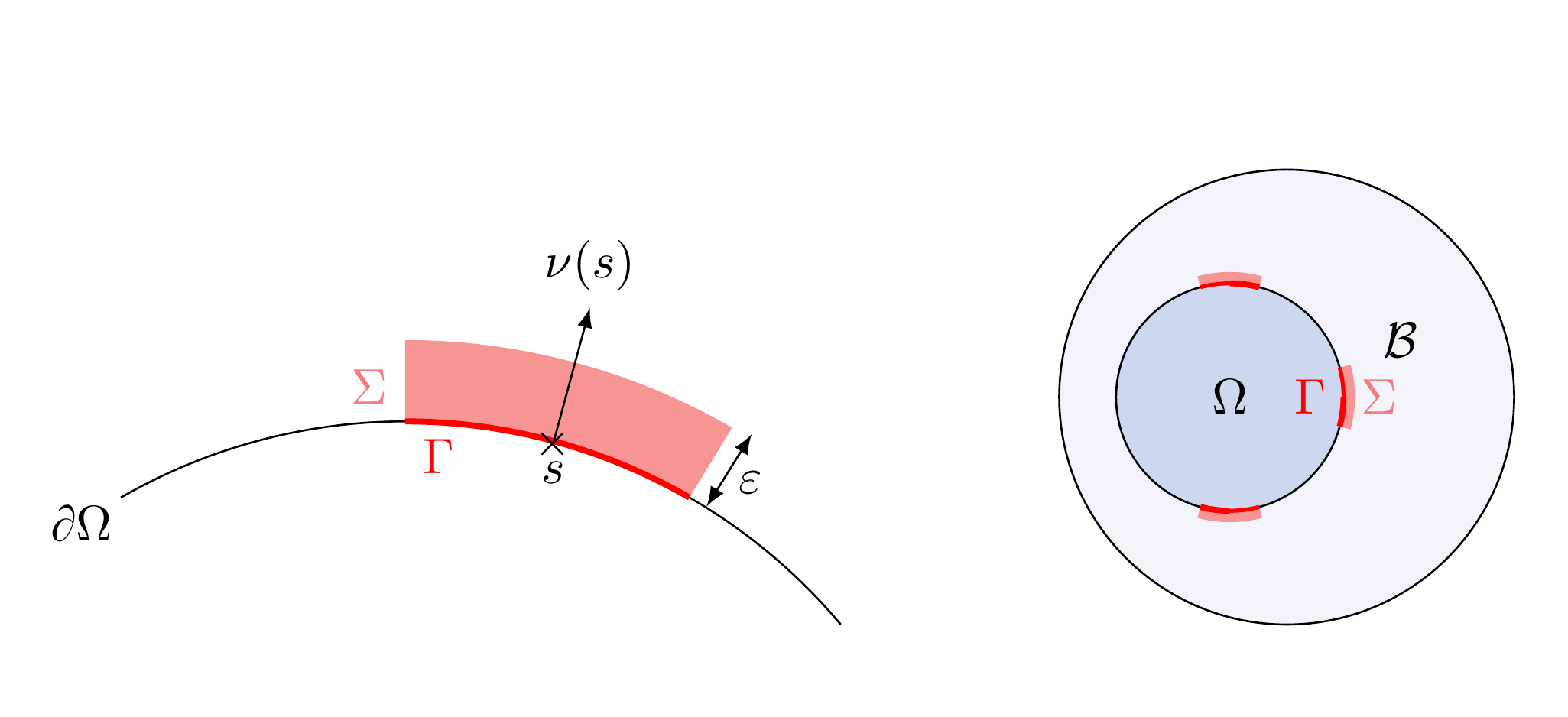} \end{center}
 \caption{The set of sensors. \label{fig1}}
\end{figure}

\begin{remark}
In this work, we choose to deal with volumetric sensors and we will determine adequately the observability term in the cost functional. Another relevant model would be to deal with boundary sensors; this would imply to  modify the choice of the aforementioned observability term in the cost functional . 
\end{remark}

When  investigating the optimal location or shape of sensors  without any restriction on the sensors domain measure, the solution is trivial and is given by $\Sigma=\widetilde{\partial\Omega}$ (see e.g. \cite{PTZCont1,PTZObs1,PTZcomplexity,PTZparab,PTZobsND}). This is not relevant for practical purposes and this is why we also assume in the sequel that the measure of the domain covered by sensors is limited: the measurable subsets $\Gamma$ of $\partial \Omega$ satisfy $\Hn(\Gamma)=L\Hn(\partial\Omega)$, where $L\in(0,1)$ denotes some given real number. Next we introduce the two classes 
\begin{equation}\label{def:VL}
\mathcal{V}_L =\lbrace a\in L^{\infty}(\partial\Omega) , \; a\in \{0,1\}, \text{ a.e. in }\partial\Omega, \ \int_{\partial\Omega} a(x)d\Hn(x) =L \Hn (\partial\Omega) \rbrace.
\end{equation}
of characteristic functions of admissible subsets $\Sigma\subset \B$ and 
\begin{equation}\label{defU}
\U = \lbrace a\in L^{\infty}(\widetilde{\partial\Omega}), \; a(s +\mu \nu (s)) = X(s), \; \text{ a.e. } (s,\mu)\in\partial\Omega \times [0,\varepsilon], \text{ with }X\in\mathcal{V}_L \rbrace 
\end{equation}
of characteristic functions of admissible subsets $\Gamma\subset \widetilde{\partial\Omega}$.

As $\Omega$ is assumed to be convex, the set $\U$ well defined. Indeed, $\Omega$ has in particular a Lipschitz boundary and the normal outward vector $\nu(\cdot)$ exists almost everywhere on $\partial\Omega$. The convexity assumption implies that the map $\widetilde{\partial\Omega}\ni s+\mu\nu (s) \mapsto (s,\mu)\in \partial\Omega\times [0,\varepsilon]$ is invertible according to the theorem of projection on a convex set.

\subsection{Criteria choices and optimization problems}\label{sec:criteriachoices}
In this subsection, we model the problem by describing the cost functional we will consider and set the optimization problems we investigate. 
Let $p_{obs}\in L^2((0,T)\times \Sigma)$ denote the pressure measured by sensors defined in $(0,T)\times \Sigma$. We extend 
$p_{obs}$ by 0 to $(0,T)\times \B$ and still denote the obtained function by $p_{obs}$, with a slight abuse of notation. Hence, one has $p_{obs} = \1e_{\Sigma} \, p_{obs}$ so that 
$p_{obs}\in L^2(0,T,L^2(\B))$. For $p_0\in H^1_0(\Omega)$, we also introduce $p_{[p_0]}$ as the solution of Problem \eqref{WaveEqn} for the initial datum $p_0$. 


The issue we address now is the following: 
\textbf{given a first (series of) measure(s), how can we determine a relevant location of sensors before performing a new (series of) measure(s)?}
\\
To achieve this goal, we use an optimal design problem as a model. 
Our approach can be split into two steps that we roughly describe.

\begin{itemize}
\item {\bf First step: determination of an initial pressure condition $p_0$.} Recall that, in the PDE model we consider (see Eq. \eqref{WaveEqn}), the initial velocity is assumed to vanish identically in $\B$. It is therefore enough to reconstruct the initial pressure. A first natural (naive) idea would be to address the problem 
\begin{equation}\label{pbOptnaive}
\inf_{p_0 \in \mathcal{P}_0(\Omega)} A_1 (\1e_\Sigma,p_0) 
\end{equation}
where 
\begin{equation}\label{defA1}
A_1 (\1e_\Sigma,p_0) := \frac{1}{2}\int_0^T \int_{\B} \1e_{\Sigma}(x) (p_{[p_0]}(t,x)-p_{obs}(t,x))^2\, dx \, dt .
\end{equation}
and $ \mathcal{P}_0(\Omega)$ is a given functional space that will be chosen adequately in the sequel.
 Unfortunately, the well-posedness of such least square problems, is in general not ensured. Moreover, because of the uncertainties of the sensors measures and the fact that $\Sigma$ is a strict subdomain of $\widetilde{\partial\Omega}$, there could exist many initial pressures $p_0$ leading to the observation $p_{obs}$. 
So, we decide to select an initial pressure function denoted $\widetilde{p_0}$ by solving a penalized optimization problem in order to impose three kinds of physical constraints: we look for (i) positive pressure term, (ii) whose support is included in a fixed compact set $K$ of $\B$, and (iii) belonging to a well-chosen functional space ensuring the well-posed character of the problem and hence, good reconstruction properties. From a practical viewpoint, this will be done by adding a penalization-regularization term denoted $\mathcal{R}(p_0)$ in the definition of the criterion. The choice of such a term will be introduced and commented in Section \ref{sec:comments}.


Solving the resulting problem (see its definition below) is a way to define an initial pressure function reconstructed (almost) everywhere in $\B$ and not only on $\Sigma$. This is the key point to address the optimal design problem in the second step. 

\item {\bf Second step: determination of the best location of sensors.} Once an initial pressure $\widetilde{p}_0$ has been determined with a given location of sensors, the new location will be obtained by solving the optimal design problem
$$
\sup_{\1e_\Sigma\in\U} A_2 (\1e_\Sigma,p_0),
$$
where
\begin{equation}\label{defA2}
A_2 (\1e_\Sigma,p_0) := \frac{\int_0^T \int_{\B} \1e_{\Sigma} (x) \partial_t p_{[\widetilde{p}_0]}(t,x)^2 \, dx \, dt}{\Vert p_0 \Vert_{H^1(\Omega)}^2 } .
\end{equation}

The  functional to maximize stands for the observation quality: we  discuss and comment this choice  below. To sum-up, we look for the location of sensors allowing the best observation of the worst possible pressure $p_0$ leading to the observation $p_{obs}$.
\end{itemize}

To summarize, this two step procedure leads to the following optimal design problem:\\

\fbox{
\begin{minipage}{0.9\textwidth}
\vspace{0.4cm}
{\sf {\bf Optimal location of sensors (new version).}\label{strategy}
Let $p_{obs}\in L^2(\Sigma)$ and $\gamma\in \R_+^*$ be a fixed parameter.
\begin{enumerate}
\item Computation of an initial pressure function $\widetilde{p}_0$ (whenever it exists) by solving the problem 
\begin{equation}\label{MainPbOpt}
\inf_{p_0 \in \mathcal{P}_0(\Omega)}J^0(\1e_\Sigma,p_0),
\end{equation} 
where the cost-functional $J^0$ is defined by 
\begin{equation}\label{J}
 J^0(\1e_\Sigma, p_0) = A_1 (\1e_\Sigma,p_0)+ \gamma \mathcal{R} (p_0),~\gamma >0
\end{equation}
and the admissible set is
$$
\mathcal{P}_0(\Omega)=\{p_0\in L^2(\Omega) \mid \operatorname{supp}(p_0)\subset K,\ p_0\geq 0\text{ a.e. in }\Omega\text{ and }\mathcal{R}(p_0)<+\infty\},
$$
where $\mathcal{R} (p_0)$ is a penalization-regularization term  whose choice will be made precise in Section \ref{sec:comments}.
\item Assuming that Problem \eqref{MainPbOpt} has a solution $\widetilde{p}_0$, determination of a new sensors location by solving 
\begin{equation}\label{ODP2}
\sup_{\1e_\Sigma\in\U} A_2(\1e_\Sigma,\widetilde{p}_0)
\end{equation}
\vspace{0.4cm}
\end{enumerate}
}
\end{minipage}
}\ \\ [0.3cm]

Next, to deal with more realistic constraints, we introduce the following modified optimal design problem, where we assume that the \textit{sensor} set is the union of $N_0$ connected components having the same length, and $N_0$ denotes a given nonzero integer. For the sake of simplicity, we will assume  that the dimension space is  $d=2$ and  the boundary of the convex set $\Omega$ is smooth, say $\mathcal{C}^1$. Let $O$ denote any point of $\Omega$ and assume that $\Omega$ has for polar equation $r=\rho(\theta)$ in a fixed orthonormal basis of $\R^2$ centered at $O$, where $\rho$ is a Lipschitz function of the one-dimensional torus $\T=\R/[0,2\pi)$.
 
Let us make this sensors set precise. Let $\ell >0$ be such that $\ell N_0 <\int_{\T}\sqrt{\rho(\theta)^2+\rho'(\theta)^2}\, d\theta$. We choose to consider the set of sensors represented by $\Sigma\subset \widetilde{\partial\Omega}$ associated to the domain $\Gamma$ with $\1e_\Gamma \in \mathcal{V}_L$ through the formula \eqref{defU}-\eqref{def:VL}, parametrized by a nondecreasing family $(\theta_n)_{n\in \{1,\dots,N_0\}}\in \T^{N_0}$ such that $\Sigma$ is associated to $\Gamma$ by the relation
\begin{equation}\label{eqSigmaGamma2}
\1e_\Sigma(x)=\1e_\Gamma(s+\mu \nu(s))=\1e_{\Gamma}(s), \qquad \text{for a.e. }x\in\widetilde{\partial\Omega}, \ s\in\partial\Omega, \ \mu\in [0,\varepsilon]
\end{equation}
for some $\1e_\Gamma\in \mathcal{V}_L$, where 
\begin{equation}\label{eqSigmaGamma2bis}
\Gamma =\bigcup_{n=1}^{N_0}\Gamma_n\qquad \text{with}\qquad \Gamma_n=\{(r\cos \theta,r\sin \theta)\in \Gamma\mid \theta\in (\theta_{n},\hat{\theta}_{n})\},
\end{equation}
where $\hat{\theta}_n$ is defined from $\theta_n$ by the relation
\begin{equation}\label{eq:perim}
\int_{\theta_n}^{\hat{\theta}_n}\sqrt{\rho(s)^2+\rho'(s)^2}\, ds=\ell .
\end{equation}
Note that, since the mapping $\T\ni \theta\mapsto \int_{\theta_n}^{\hat{\theta}_n}\sqrt{\rho(s)^2+\rho'(s)^2}\, ds$ is continuous and monotone increasing, it defines a bijection. This justifies the consistence of definition of $\hat{\theta}_n$ by \eqref{eq:perim}.

Roughly speaking, this last equality imposes that the set of sensors is represented by $N_0$ almost identical connected components (up to isometries). 

To avoid the superposition of sensors, we will also impose that
\begin{equation}\label{constraint:sens}
\forall n\in \{1,\dots,N_0-1\}, \qquad \hat{\theta}_{n}\leq \theta_{n+1}.
\end{equation}
It will be useful for the forthcoming analysis to notice that this last condition  rewrites 
\begin{equation}\label{constraint:sensbis}
\forall n\in \{1,\dots,N_0-1\}, \qquad \int_{\theta_n}^{\theta_{n+1}}\sqrt{\rho(s)^2+\rho'(s)^2}\, ds\geq \ell.
\end{equation}

\noindent As a conclusion, we will restrict the set of admissible configurations to 
\begin{equation}\label{def:UN0eps}
\mathcal{U}_{L,N_0}^{\ell}=\{\1e_\Sigma\in\U\mid \Sigma\text{ associated to $\Gamma$ by \eqref{eqSigmaGamma2} with }\1e_{\Gamma}\in \mathcal{V}_{L,N_0}^{\ell}\},
\end{equation}
where
\begin{equation}\label{def:VN0eps}
\mathcal{V}_{L,N_0}^{\ell}=\{\1e_\Gamma\in\V\mid \Gamma\text{ satisfies \eqref{eqSigmaGamma2bis}-\eqref{eq:perim}-\eqref{constraint:sensbis}}\}.
\end{equation}

\fbox{
\begin{minipage}{0.9\textwidth}
\vspace{0.4cm}
{\sf {\bf Optimal location of sensors (updated version with a maximal number of connected components).}
Let $N_0\in \N^*$, $\ell >0$ such that $\ell N_0 <\int_{0}^{2\pi}\sqrt{\rho(\theta)^2+\rho'(\theta)^2}\, d\theta$, $p_{obs}\in L^2(\Sigma)$ and $\gamma\in \R_+^*$ be a fixed parameter.
\begin{enumerate}
\item Computation of an initial pressure function $\widetilde{p}_0$ (whenever it exists) by solving the problem \eqref{MainPbOpt} (as for the optimal design problem \eqref{ODP2}).
\item Assuming that Problem \eqref{MainPbOpt} has a solution $\widetilde{p}_0$, determination of a new sensors location by solving 
\begin{equation}\label{ODP3}
\sup_{\1e_\Sigma\in\mathcal{U}_{L,N_0}^{\ell}} A_2(\1e_\Sigma,\widetilde{p}_0),
\end{equation}
where $\mathcal{U}_{L,N_0}^{\ell}$ is given by \eqref{def:UN0eps}.
\vspace{0.4cm}
\end{enumerate}
}
\end{minipage}
}\ \\ [0.3cm]

In Section \ref{sec:num}, we will discuss this method and illustrate it with the help of several numerical results. 

\bigskip

\paragraph{Comments on the terms $A_1$ and $A_2$.}
We end this section with comments  on the choice of the functional. 

The term $ A_1(\1e_\Sigma,p_0)$ is a \textit{least-square } fidelity term, ensuring that the initial pressure $p_0$ makes the pressure $p_ {[p_0]}$ (corresponding to the reconstructed image) as close as possible to the observed pressure $p_{obs}$ on the domain occupied by sensors. 
The term $A_2(\1e_\Sigma,p_0)$ is inspired by the notion of {\it observability} in control or inverse problems theory.
Indeed, the equation \eqref{WaveEqn} is said {\it observable in time $T$ on $\Sigma$} if there exists a positive constant $C$ such that the inequality 
\begin{equation}\label{ineqobs}
C \Vert p_0\Vert_{H_0^1}^2
\leq \int_0^T\int_\B \1e_\Sigma (x) \left| \partial_t p_{[p_0]} (t,x)\right|^2 \,dx dt,
\end{equation}
holds for every initial datum $p_0 \in H_0^1(\Omega)$. 
 The largest constant $C=C_T(\1e_\Sigma)$ such that the inequality \eqref{ineqobs} holds is the so-called {\it observability constant} and writes
$$
C_T(\1e_\Sigma)=\inf_{p_0\in H^1_0(\B)}\frac{\int_0^T\int_\B \1e_\Sigma (x) \left| \partial_t p_{[p_0]} (t,x)\right|^2 \,dx\, dt}{\Vert p_0\Vert_{H_0^1}^2}.
$$
This constant $C_T(\1e_\Sigma)$ provides an overview of the well-posedness character of the reconstruction of $p_0$ from measurements of $p _{[p_0]}$ on $\Sigma$. Roughly speaking, the largest the observability constant is, the best is the reconstruction quality.

The functional $A_2$ thus models the efficiency of a set of sensors occupying the domain $\Sigma$.
Note in addition, that the choice of the observation variable $\1e_\Sigma \partial_t p_{[p_0]}$ is driven by the fact that sensors are 
 generally piezoelectric microphones, assumed to record time variations of the pressure.

Eventually, to understand the role of the maximization with respect to the location of sensors, we have illustrated a \textit{bad }situation on Figure \ref{Illu1}; here a location of sensors on a domain $\Sigma_2$ provides a less accurate result than another location on a subdomain $\Sigma_1$, since the observed pressure coincides only with the true one on the domain $\Sigma_2$. The maximization with respect to the design can then be interpreted as a \textit{worst case} functional, used to move the sensors toward a location ignored by the previous measures and where the reconstructed pressure is far from the observed one. 

\begin{figure}[H]
\begin{center}
%
 \includegraphics[width=0.8 \linewidth]{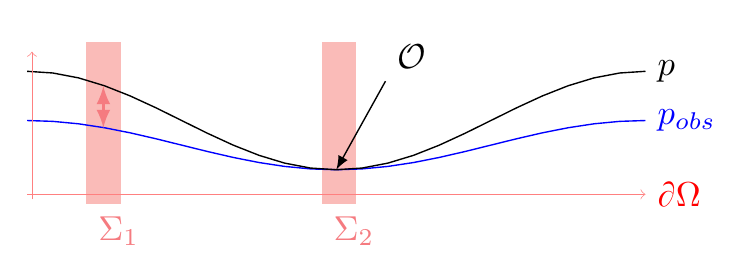} 
 \caption{A bad situation.\label{Illu1}}
\end{center}
\end{figure}

\section{Analysis of Problems \eqref{MainPbOpt} and \eqref{ODP2}}\label{sec:Penal}

\subsection{Choice of the regularization term $\mathcal{R}(p_0)$}\label{sec:comments}
Let us add some comments to make the model (via the cost functional) fully precise. 
If we consider Problem \eqref{pbOptnaive}, it is clear that the ill-posedness of this problem prevents to find a solution (to highlight such a claim, one can easily build counterexamples based on gaussian beams). Therefore, 
we have to define a cost functional that is coercive with respect to the functional space by adding a penalization-regularization term $\mathcal{R}(p_0)$.
\begin{itemize}
\item If we decide to choose the $p_0$ space as $L^2(\mathcal{B})$ we ask for the weakest regularity. In that case $\mathcal{R}(p_0) = \|p_0\|_{L^2}$. This implies that $p_0$ will be small since the first term only involves ``sparse'' observation of $p$ and should be small even with small $\varepsilon$.
\item We hope to recover heterogeneities as images with sharp edges.The more appropriate functional space for that purpose is the space of functions of bounded variation.
Indeed, involving the first derivative via the total variation has a denoising effect.
However, considering the solution of the wave equation with data in $BV$ is not standard and (at least mathematically) and it will be more convenient for the upcoming analysis to deal with more regular data. 
We choose a penalization implying in particular a $H^1_0$-smoothing effet. Therefore, we expect to recover diffuse objects. To avoid a too hard regularization process we also use the $W^{1,1}$ norm (which is the same as the $TV$-one in that case). 
 \item 
Yet, we don't have any compactness result to achieve the convergence process of subsequences. Therefore, we add a (small) viscosity penalization term which must be viewed as a theoretical  tool.
\end{itemize}
Eventually, we choose 
\begin{equation}\label{def:R0}
\boxed{\mathcal{R}(p_0)=\|\nabla p_0\|_{L^1} +\frac{\varepsilon}{2\gamma} \|p_0\|_{H^1}^2},
\end{equation}
so that the cost functional reads 
\begin{equation}\label{cost1}
J(p_0):= A_1(\1e_\Sigma,p_0) + \gamma \|\nabla p_0\|_{L^1} +\frac{\varepsilon}{2} \|p_0\|_{H^1}^2,
\end{equation}
where $\varepsilon << \gamma$
and $ A_1(p_0) =A_1(\1e_\Sigma,p_0)$ is given by \eqref{defA1}. 
Throughout this section, we omit the dependence with respect to $ \Sigma$ since this set is fixed. 
Hence, Problem \eqref{MainPbOpt} can be written as
\begin{equation}\label{MainPbOpt2}
\inf_{p_0 \in \mathcal{P}_K(\Omega)}J( p_0),
\end{equation} 
where $ K$ is a fixed compact subset of $\Omega$ and 
$$\mathcal{P}_K(\Omega)=\{p_0\in H^1_0\Omega) \mid \operatorname{supp}(p_0)\subset K\text{ and }p_0\geq 0\text{ a.e. in }\Omega\}.$$
Note that $\mathcal{P}_K(\Omega)$ is convex and closed for the weak $H^1(\Omega) $-topology.
{\rem We can choose for instance 
 $ K:=\{ x \in \Omega ~|~d(x, \partial \Omega ) \ge \delta~\} $ for some positive parameter $\delta$.} 

\subsection{Analysis of Problem \eqref{MainPbOpt2}}
In the sequel we denote by bold letters functions that are defined on $(0,T) \times \mathcal{B}$. Precisely, 
for any $p\in H^1_0(\Omega)$ , we denote $\bf{p} $ the solution of \eqref{WaveEqn} where $\bf{p}(0, \cdot)$=$ p$. 

\paragraph{Existence and uniqueness.}

{\theorem Fix $\varepsilon >0$ and $\gamma > 0$. Then problem \eqref{MainPbOpt2} has a unique solution $p^*$.}
\begin{proof}
Let $(p_n)_{n\in\N}$ be a minimizing sequence of \eqref{MainPbOpt2}. Therefore, $(p_n)_{n\in\N}$ is weakly convergent to some $p_0$ in $ H^1(\W)$ and $W^{1,1}(\W)$ and strongly in $L^2(\W)$. As $\mathcal{P}_K(\Omega)$ is $ H^1(\W)$ weakly closed then $p_0 \in \mathcal{P}_K(\Omega)$. Moreover, the sequence $({\bf p_n})_{n\in\N}$ defined by 
${\bf p_n}:=p_{[p_n]} $ strongly converges to ${\bf p_0}:=p_{[p_0]} $ in $L^\infty\left([0,T], L^2(\B)\right)$ (see \cite[Theorem 2 p. 567]{DL}). According to the Lebesgue theorem, we infer that $ A_1(p_n)\to A_1(p_0)$ as $n\to +\infty$. We end the proof with the lower semi-continuity of the norms in the regularizing term. Uniqueness is an easy consequence of the strict convexity of the $H^1$-norm.
Note  that the $\|p_0\|_{H^1}^2$ term is used both for existence and uniqueness.
\end{proof}

\paragraph{Optimality conditions.} 
Let $p^*$ be the solution to \eqref{MainPbOpt2} and (as before) ${\bf p^*} $ the solution of \eqref{WaveEqn} where ${\bf p^*}(0, \cdot) = p^*$ and look for optimality conditions. For every $p \in \mathcal{P}_K(\Omega)$ we get 
$$ 0 \in \partial J(p^*) ,$$
where $\partial J(p^*) $ stands for the subdifferential of $J$ at $p^*$ (see \cite{ET} for example). Indeed, $J$ is not G\^ateaux-differentiable because of the TV term. This gives 
\begin{equation} \label{CN1} -D [A_1](p^*)- \varepsilon (p^*-\Delta p^*) \in \gamma \partial TV(p^*) .
\end{equation} 
Recall that $TV(p^*) = \|\nabla p^*\|_{L^1}$ since $ p^*\in H^1_0(\W)$ but we use this notation for convenience. 
We first compute $D [A_1](p^*) $ (the derivative with respect to $p^*$) introducing an adjoint state. Let us define ${\bf q^*}$ as follows :
\begin{equation}\label{AdjWaveEqn} \left\lbrace \begin{array}{ll}
\partial_{tt} {\bf q^*} (t,x)\! - \!\operatorname{div} (c(x) \nabla {\bf q^*} (t,x))\!= \!({\bf p^*}(t,x)\!-\!p_{obs}(t,x)) )\1e_{\Sigma}(x)& \text{in }(0,T)\times \B,\\
{\bf q^*}(T,\cdot) = 0, ~
\partial_t {\bf q^*}(T,\cdot) =0 & \text{on } \B, \\
{\bf q^*}=0 & \text{on } (0,T)\times \partial\B,
\end{array}\right. \end{equation}
For every $p \in H^1_0(\W)$, a simple computation yields 
$$D[ A_1](p^*)\cdot p = \int_0^T \int_{\B} \1e_{\Sigma}(x) ({\bf p^*} (t,x)-p_{obs}(t,x)) \, {\bf p}(t,x) dx \, dt .$$
Indeed the derivative of $p_0 \mapsto p_{[p_0]} $ with respect to $p_0$ is $ p_{[p_0]} $ because of the linearity of the equation. Using \eqref{AdjWaveEqn} and two integrations by parts, we get
$$ D[ A_1](p^*)\cdot p = - \int_{\B} \1e_{\Sigma}(x) \,{\bf q^*}(0,x) p(x) \, dx~,$$
so that $D[ A_1](p^*) = - \1e_{\Sigma}\, {\bf q^*}(0).$ Equation \eqref{CN1} writes
\begin{equation} \label{CN2} \frac{1}{\gamma}
 \Big ( \1e_{\Sigma} {\bf q^*}(0) - \varepsilon (p^*-\Delta p^*) \Big ) \in \partial TV(p^*) .
\end{equation} 
We finally obtain the following
\begin{theorem}\label{thm:optcondpb1}
 Fix $K,~\Sigma,~\varepsilon >0$ and $\gamma>0$. The function $p^* \in H^1_0(\W)$ is the optimal solution to \eqref{MainPbOpt2} if and only if equation \eqref{CN2} is satisfied with ${\bf q^*} $ solution to 
\eqref{AdjWaveEqn} and ${\bf p^*} $ solution to 
\eqref{WaveEqn} with $p^*$ as initial condition.
\end{theorem}
The optimality condition above is in fact sufficient by an easy convexity argument.
\\
{\rem The computation of $ \partial TV(p^*)$ is standard in a finite dimensional setting, either using a primal-dual algorithm or performing an approximation of $TV(p^*)$ }. 
\subsection{ Analysis of Problem \eqref{ODP2}}\label{sec:existODP}
We now focus on the \textit{optimal design} problem \eqref{ODP2}. To perform the analysis, we give an equivalent formulation of $A_2(\1e_\Sigma,p_0)$ using the particular form of $\1e_\Sigma$. Indeed, for all $x\in \Sigma$, we will say that $\Gamma$ is a subset of $\partial\Omega$ associated to $\Sigma\in \U$ whenever
\begin{equation}\label{SigmaGamma}
\1e_\Sigma(x)=\1e_\Gamma(s+\mu \nu(s))=\1e_{\Gamma}(s), \qquad \text{for a.e. }x\in\widetilde{\partial\Omega}, \ s\in\partial\Omega, \ \mu\in [0,\varepsilon]
\end{equation}
for some $\1e_\Gamma\in \mathcal{V}_L$ where $\1e_\Gamma$ is the characteristic function of a subset $\Gamma$ of $\partial\Omega$. 
Using Fubini-Tonelli-Lebesgue theorem, one has
\begin{eqnarray}
A_2(\1e_\Sigma,p_0)&=& \frac{1}{\Vert p_0 \Vert_{H^1(\Omega)}^2} \int_0^T \int_{\B} \1e_{\Sigma} (x) \partial_t p_{[p_0]}(t,x)^2 \, dx \, dt\nonumber \\
&=& \frac{1}{\Vert p_0 \Vert_{H^1(\Omega)}^2} \int_{\partial\Omega} \1e_{\Gamma} (s)\psi_{[p_0]}(s) \, d\Hn \label{n2201}
\end{eqnarray}
where 
\begin{equation}\label{n2202}
\psi_{[p_0]}(s)=\int_0^T \int_0^\varepsilon \partial_t p_{[p_0]}(t,s+\mu\nu(s))^2\, d\mu\, dt
\end{equation}
Indeed, it follows directly from the regularity property of $p_{[p_0]}$ that $\psi_{[p_0]}\in L^1(\partial\Omega)$.
 Next theorem gives an existence result: 

\begin{theorem}
Let $\widetilde{p}_0 \in H^1(\Omega)$ be a solution of Problem \eqref{MainPbOpt}. Then, the optimal design problem \eqref{ODP2} has at least one solution.
Moreover, there exists a real number $\lambda$ such that $\Sigma^*$ is a solution of \eqref{ODP2} if and only if $\Sigma^*$ is associated to $\Gamma^*$ (in the sense of \eqref{SigmaGamma}) where
\begin{equation}\label{omegastar}
\1e_{\{\psi_{[\widetilde{p}_0]}(s) > \lambda\}} \leq \1e_{\Gamma^*}(s) \leq \1e_{\{ \psi_{[\widetilde{p}_0]}\geq \lambda\}}(s),\qquad \text{for a.e. }s\in \partial\Omega .
\end{equation}
\end{theorem}
\begin{proof}
The proof is a direct adaptation of \cite[Theorem 1]{PTZcomplexity}. It lies directly upon the fact that the functional $A_2$ rewrites as \eqref{n2201}, as well as a standard argument of decreasing rearrangement. Another convexity argument can be used to get this result, observing that the functional
$$
\tilde A_2: L^\infty(\partial\Omega,[0,1])\ni \rho \mapsto \int_{\partial\Omega} \rho (s)\psi_{[p_0]}(s) \, d\Hn
$$
is continuous for the weak-star topology of $L^\infty$ and that the set
$$
\mathcal{C}_L=\left\{\rho \in L^\infty(\partial\Omega,[0,1])\mid \int_{\partial\Omega} \rho =L\Hn(\partial\Omega)\right\}
$$
is compact for this topology. It follows that the problem 
$$
\inf \left\{ \int_{\partial\Omega} \rho (s)\psi_{[p_0]}(s) \, d\Hn, \ \rho \in \mathcal{C}_L\right\}
$$
has at least one solution. Moreover, one shows easily that the solution can be chose among the extremal points of the convex set $\mathcal{C}_L$, by using the convexity of the mapping $\tilde A_2$.
\end{proof}

\subsection{Analysis of Problem \eqref{ODP3}}\label{sec:existODP3}

This section is devoted to the resolution of  Problem \eqref{ODP3} in the two-dimensional case, where the considered sensors are the union of $N_0$ \textit{similar} connected components (the definition of the sensors set has been precised with \eqref{def:UN0eps}).
%
Let us define
\begin{equation}\label{def:fpsi}
f_{[\widetilde{p}_0]}:\T\ni \theta\mapsto \psi_{[\widetilde{p}_0]}(\rho(\theta)\cos \theta,\rho(\theta)\sin\theta).
\end{equation}
where $\psi_{[\widetilde{p}_0]}$ is given by \eqref{n2202}.

\begin{theorem}\label{theo:ODP3}
Let $\widetilde{p}_0 \in H^1(\Omega)$ be a solution of Problem \eqref{MainPbOpt}. Then, the optimal design problem \eqref{ODP3} has at least a solution $\Sigma^{N_0,\ell}$. 
Let $\Gamma^{N_0,\ell}$ (resp. $(\theta_n)_{n\in \{1,\dots,N_0\}}\in \T^{N_0}$) be the associated set of $\partial\Omega$ in the sense of \eqref{eqSigmaGamma2} (resp. the associated family of angles in the sense of \eqref{eqSigmaGamma2bis}). Then, 
\begin{enumerate}
\item if $N_0=1$, then, $f_{[\widetilde{p}_0]}(\theta_1)=f_{[\widetilde{p}_0]}(\hat{\theta}_1)$.
\item if $N_0\geq 2$, let $n\in \{1,\dots,N_0-1\}$. One has the following alternative: either $\hat{\theta}_n=\theta_{n+1}$, or $f_{[\widetilde{p}_0]}(\theta_n)=f_{[\widetilde{p}_0]}(\hat{\theta}_n)$. 
\end{enumerate}
If one assumes furthermore that $\rho$ belongs to $C^2(\T)$, then one has
$$
\frac{f_{[\widetilde{p}_0]}'(\hat{\theta}_n)}{\sqrt{\rho(\hat{\theta}_n)^2+\rho'(\hat{\theta}_n)^2}}-\frac{f_{[\widetilde{p}_0]}'(\theta_n)}{\sqrt{\rho(\theta_kn^2+\rho'(\theta_n)^2}}\leq 0
$$ 
whenever $N_0=n=1$ or $N_2\geq 2$ and $n\in \{1,\dots,N_0-1\}$ is such that $\hat{\theta}_n<\theta_{n+1}$.
\end{theorem}
\begin{proof}
We first assume, without loss of generality that $N_0\geq 2$, the case where $N_0=1$ being easily inferred by adapting the following reasoning. In what follows, we will use several times that $\rho(\theta)>0$ for all $\theta\in \T$, which is a consequence of the fact that the point $O$ (center of the considered orthonormal basis) belongs to the open set $\Omega$.

First, using the same computations as those at the beginning of Section \ref{sec:existODP}, we claim that Problem \eqref{ODP3} can be recast as
\begin{equation}\label{ODP3bisbis}
\sup_{(\theta_1,\dots,\theta_{N_0})\in \Theta_{L,N_0}^\ell}J(\theta_1,\dots,\theta_{N_0})
\end{equation}
where
\begin{equation}
J(\theta_1,\dots,\theta_{N_0})=\sum_{n=1}^{N_0}\int_{\theta_{n}}^{\hat{\theta}_{n}}f_{[\widetilde{p}_0]}(\theta)\sqrt{\rho(\theta)^2+\rho'(\theta)^2}\ d\theta,
\end{equation}
with $f_{[\widetilde{p}_0]}$ defined by \eqref{def:fpsi}, $\hat{\theta}_n$ defined by \eqref{eq:perim} and
\begin{equation}
\Theta_{L,N_0}^\ell \!=\left\{(\theta_1,\dots,\theta_{N_0})\in \T^{N_0}\mid \int_{\theta_n}^{\theta_{n+1}}\!\!\!\!\sqrt{\rho(s)^2+\rho'(s)^2}\, ds\geq \ell, ~ n=1,\dots,N_0-1\right\}.
\end{equation}
Existence of a solution of this problem is standard and follows immediately from both the compactness of $\Theta_{L,N_0}^\ell $ in $\T^{N_0}$ and the continuity of the functional.
Let $(\theta_1,\dots,\theta_{N_0})\in \Theta_{L,N_0}^\ell$ be a solution of the optimization problem \eqref{ODP3bisbis}. Let us assume the existence of $n\in \{1,\dots,N_0-1\}$ such that $\hat{\theta}_n<\theta_{n+1}$, in other words such that
\begin{equation}\label{constraint:1020}
\int_{\theta_n}^{\theta_{n+1}}\sqrt{\rho(s)^2+\rho'(s)^2}\, ds> \ell.
\end{equation}

Then, it follows from the Karush-Kuhn-Tucker theorem that $\frac{\partial J}{\partial\theta_n}(\theta_1,\dots,\theta_{N_0})=0$, which rewrites
$$
\sqrt{\rho(\theta_n)^2+\rho'(\theta_n)^2}f_{[\widetilde{p}_0]}(\theta_n)=f_{[\widetilde{p}_0]}(\hat{\theta}_n)\sqrt{\rho(\hat{\theta}_n)^2+\rho'(\hat{\theta}_n)^2}\frac{\partial\hat{\theta}_n}{\partial\theta_n}(\theta_n),
$$
simplifying into
\begin{equation}\label{eq2055}
f_{[\widetilde{p}_0]}(\theta_n)=f_{[\widetilde{p}_0]}(\hat{\theta}_n),
\end{equation}
by using that 
\begin{equation}\label{dhattheta}
\frac{\partial\hat{\theta}_n}{\partial\theta_n}(\theta_n)=\frac{\sqrt{\rho(\theta_n)^2+\rho'(\theta_n)^2}}{\sqrt{\rho(\hat{\theta}_n)^2+\rho'(\hat{\theta}_n)^2}},
\end{equation}
according to the combination of \eqref{eq:perim} with the implicit functions theorem. Furthermore, assuming that $\rho$ is a $C^2$ function, the necessary second order optimality conditions write
$$
\frac{\partial^2 J}{\partial\theta_n^2}(\theta_1,\dots,\theta_{N_0})\leq 0,
$$
which comes to
\begin{multline}
-\frac{\rho'(\theta_n)\left(\rho(\theta_n)+\rho''(\theta_n)\right)}{\sqrt{\rho(\theta_n)^2+\rho'(\theta_n)^2}}f_{[\widetilde{p}_0]}(\theta_n)-f_{[\widetilde{p}_0]}'(\theta_n)\sqrt{\rho(\theta_n)^2+\rho'(\theta_n)^2}\\
+\frac{\partial\hat{\theta}_n}{\partial\theta_n}(\theta_n)f_{[\widetilde{p}_0]}'(\hat{\theta}_n)\sqrt{\rho(\theta_n)^2+\rho'(\theta_n)^2}+\frac{\rho'(\theta_n)\left(\rho(\theta_n)+\rho''(\theta_n)\right)}{\sqrt{\rho(\theta_n)^2+\rho'(\theta_n)^2}}f_{[\widetilde{p}_0]}(\hat{\theta}_n)\leq 0\label{eq2054}
 \end{multline}
Combining this inequality with both the optimality condition \eqref{eq2055} and the relation \eqref{dhattheta} yield the desired result.

\end{proof}

 \section{Numerical simulations}\label{sec:num}
 
We are now interested in the numerical resolution of the problems following the strategy provided by \eqref{MainPbOpt}-\eqref{ODP2}. However, as explained previously, we are going to perform some simplification of the model in view of the numerical implementation. \\

For instance, we assume here that the velocity $c(\cdot)$ is constant and equals to $1$. 
Noting that its influence has not a major importance in the methodology and since the purpose of this article is to validate the methodology introduced in section \ref{sec:criteriachoices}, this simplification allows us to compute in a simple way the solution of the wave equation in the Fourier space and then  obtain a high resolution in time and space.
 
The minimization of functional $J^0$ (Problem \eqref{MainPbOpt})
is performed by using a splitting gradient step descent. More precisely, the idea is to alternate an explicit treatment 
of $A_1(1_{\Sigma},p_0) $ via a time reversal imaging \cite{bretinTR-ContempMath,Time_reversal_BLP,Timereversal2,Timereveral1}, 
and an implicit treatment of the $TV$ term thanks to the algorithm introduced by A. Chambolle in \cite{chambolleTV}. 
Moreover, in practice, we also apply the FISTA strategy \cite{Beck_teboulle} to accelerate 
the convergence of this splitting iterative gradient scheme. 
 
About the determination of the best location of sensors $\Gamma^*$, the idea is to compute 
the energy function $\psi_{[p_0]}$ defined for all $s \in \partial \Omega$ by 
$$ \psi_{[p_0]}(s) = \int_0^{T} \partial_{t} p_{[p_0]}(t,s)^2 dt.$$
\paragraph{$\bullet$ For Problem \eqref{ODP2} (continuous setting):}
We  apply Theorem 3.3, which shows that $\Gamma^*$ is expected to have the form 
$$\Gamma^{*}_{\lambda} = \{ s \in \partial \Omega \mid \psi_{[p_0]}(s) \geq \lambda\}.$$
It is also a simplification since $\psi_{[p_0]}$ is observed only on $\partial \Omega$ and 
the optimal constant $\lambda$ is determined by using a dichotomy approach,   looking for $\lambda$ such that the constraint $\Hn(\Gamma^{*}_{\lambda}) = L \Hn(\partial \Omega)$ be satisfied. 

\paragraph{$\bullet$ For Problem \eqref{ODP3} (discrete setting, with a maximal number of connected components):}
although Theorem \ref{theo:ODP3} provides a partial characterization of the solution, we observed numerically the existence of numerous local optima. For this reason, we chose to implement a genetic algorithm procedure.
More precisely, we solve Problem \eqref{ODP3bisbis} by using the Matlab function \textsf{ga} allowing to solve a finite dimensional optimization problem with a genetic algorithm. It is notable that, practically speaking, it is not worth to take into account the constraint \eqref{constraint:1020} since it is naturally satisfied at the optimum.

\bigskip

We  next recall several standard results about {\it time reversal Imaging} in Section \eqref{subsect_TR}. 
Moreover, in Section \eqref{subsect_GradientJO}, we show how the gradient of the functional $A_1(1_{\Sigma},p_0)$ can be computed thanks to a generalized time reversal imaging technique.
We also presents some classical discretization of the wave equation and its integration in Fourier space 
\eqref{sec:discretization_wave_equation}. Finally, some numerical illustrations of our approach are 
presented in Section \ref{sec:numerical_experiment}, which highlights the improved reconstruction of the source 
$p^0$ when we optimizing the location of the sensors. 

 \subsection{Time reversal imaging} \label{subsect_TR}

Recall that $p_{[p_0]}$ denotes the solution of Problem \eqref{WaveEqn} for the initial datum $p_0$.
If the given data $p_{obs}$ are complete on the boundary of $\Omega$ i.e
$$
p_{obs}(t,y) = p_{[p_0]}(t,y)\quad \text{ for all }(t,y) \in [0,T]\times \partial \Omega,
$$ 
then, the reconstruction of the initial source $p_0$ from data $g = p_{obs}$ can be done following time reversal imaging method, by using that 
 $$ p_0(\cdot) \simeq \I[p_{obs}](\cdot) = w(T,\cdot),$$
 where $w$ is defined as the solution of the (backward) wave equation
 $$
 \begin{cases}
 \partial_{tt} w(t,x) - \Delta w(t,x) = 0, & (t,x) \in [0,T] \times \Omega, \\
 w(0,x) = \partial_t w(0,x) = 0, & x \in \Omega, \\
 w(t,y) = p_{obs}(y,T-t), & t \in [0,T].
 \end{cases}
 $$
 More precisely,  $T$ is required to be sufficiently large to satisfy $u(T,\cdot) \simeq 0$ and $\partial_t u(T,\cdot) \simeq 0$ on $\Omega$ \cite{Timereversal2}.
 However, as explained in \cite{bretinTR-ContempMath}, the discretization of this imaging functional requires  data interpolation on the boundary of $\Omega$: this introduces smoothing effects on the reconstructed image (identical to the use of the penalization term $\mathcal{R}(p_0)$  in Problem \eqref{pbOptnaive}). In practice, it is more efficient to use an approximation version reading 
 $$\I[p_{obs}](x) = \int_{0}^{T} v_s(T,x) ds,$$
 where $v_s$ solves the wave equation 
 $$ \begin{cases}
 \partial^2_{tt} v_s(t,x) - \Delta v_{s}(t,x) = \partial_t \left( \delta_{\{t=s\}} g(x,T-s) \right) \delta_{\partial \Omega}, \quad (t,x) \in \R \times \R^d  \\
 v_s(t,x) = 0, \quad \partial_t v_s(t,x) = 0, \quad x \in \R^d, t<s.
 \end{cases}
 $$
 Here, $\delta_{\{t=s\}}$ denotes the time Dirac distribution at time $t=s$ and $\delta_{\partial \Omega}$ is the surface Dirac measure on the manifold $\partial\Omega$. 
  
 In particular, by using the so-called Helmholtz-Kirchhoff identity, it is proved in \cite{ammari2008introduction} 
 that when $\Omega$ is close to a sphere with large radius in $\R^d$, there holds
 $$p_0(x) \simeq \I[p_{obs}](x).$$
 
 An other advantage of this modified time reversal imaging technique is its variational character. Indeed, recall that 
 $p_{[p_0]}$ can be expressed as 
 $$ p_{[p_0]}(\cdot) = \partial_t \mathbf{G}(t,\cdot)*p_0,$$
 where $*$ is the convolution product in space,
 $\mathbf{G}$ the temporal Green function obtained as the inverse Fourier transform of $ {G}_{\omega}$ 
 $$\mathbf{G}(t,\cdot) = \FF^{-1}_t [G_{\omega}(\cdot)](t),$$
 where $G_{\omega}$ denotes the outgoing fundamental solution to the Helmholtz operator $- (\Delta + \omega^2)$ in $\R^d$, that is the distributional solution of the equation
 $$ 
 (\Delta + \omega^2) G_{\omega}(x) = -\delta_{\{x=0\}} \qquad x\in \R^d
 $$
 subject to the outgoing Sommerfeld radiation equation 
 $$ \lim _{{|x|\to \infty }}|x|^{{{\frac {d-1}{2}}}}\left({\frac {\partial }{\partial |x|}}-i\omega\right)u(x)=0.$$
The discrepancy functional $A_1$ defined by \eqref{defA1} can then be recast as
 $$ A_1(\1e_\Sigma,p_0) = \frac{1}{2} \int_{0}^{T} \int_{\B} \1e_{\Sigma}(y) ( (\partial_t \mathbf{G}(t,.)*p_0(.))(y) - p_{obs}(t,y))^2 dy \,dt.$$
Moreover, one shows easily that its G\^ateaux-derivative with respect to the variable $p_0$, defined by 
$$
\langle dA_1(\1e_\Sigma,p_0) ,h\rangle=\lim_{\tau \searrow 0}\frac{A_1(\1e_\Sigma,p_0 + \tau  h)-A_1(\1e_\Sigma,p_0)}{\tau }
$$
writes 
$$
\langle dA_1(\1e_\Sigma,p_0) ,h\rangle=\int_{\Omega}\nabla A_1(\1e_\Sigma,p_0) (x) h(x)\, dx,
$$
where $\nabla A_1(\1e_\Sigma,p_0)$ is the gradient with respect to $p_0$, identified to 
 $$ \nabla A_1(\1e_\Sigma,p_0) = \int_{0}^{T} \partial_t \mathbf{G}(t,.)*\left[ (p_{[p_0]}(t,\cdot) - p_{obs}(t,\cdot)) \1e_\Sigma(\cdot) \right] dt$$
or
 \begin{equation}\label{nablaA1}
 \nabla A_1(\1e_\Sigma,p_0) = \int_{0}^{T} v_s(T,\cdot)\, ds,
 \end{equation}
 where $v_s$ is solution of 
 $$ \begin{cases}
 \partial^2_{tt} v_s(t,x) - \Delta v_{s}(t,x) = \partial_t \left( \delta_{\{t=s\}} \left[ p_{[p_0]}(T-s,\cdot) - p_{obs}(T-s,\cdot) \right] \right) \1e_\Sigma(\cdot), \\
 v_s(t,x) = 0, \quad \partial_t v_s(t,x) = 0, \quad x \in \R^d, t<s.
 \end{cases}
 $$
This claim follows from an straightforward adaptation of the proof of Theorem \ref{thm:optcondpb1}.
In particular, this shows that the gradient $\nabla A_1(\1e_\Sigma,p_0)$ corresponds to the modified 
time reversal imaging associated to the data $p_{[p_0]} - p_{obs}$,
 where the Dirac mass $\delta_{\partial \Omega}$ is replaced by the characteristic function $\1e_\Sigma$ 
 and 
 $$ \nabla A_1(\1e_\Sigma,p_0) = \I[p_{[p_0]} - p_{obs}].$$
 

\subsection{Solving Problem \eqref{MainPbOpt}} \label{subsect_GradientJO}
In this section, we focus on numerical algorithms to solve Problem \eqref{MainPbOpt}. Because of the smoothing effects mentioned in Section \ref{subsect_TR}, it is not worth to add a $H^1$-penalization term from a practical point of view. This is why we eventually consider that
$$
 J^0(\1e_\Sigma, p_0) = A_1 (\1e_\Sigma,p_0)+ \gamma TV(p_0).
$$

A first idea is to use a gradient-iterative scheme with an implicit treatment of the $TV$ norm, combined with an explicit treatment of $A_1(1_{\Sigma},p_0)$.
In this context, the simplest iterative scheme is the forward-backward algorithm, reading 
$$ p_0^{n+1} = (I + \eta\gamma \partial TV )^{-1} (p_0^n - \eta\nabla A_1(\1e_{\Sigma},p_0^n)), \qquad n \geq 0,$$
where $p^0 = 0$ and $\eta$ is a given (small positive) descent step.
The $TV$ proximal operator $\operatorname{prox}_{\eta TV}[u]$  is defined by 
$$ \operatorname{prox}_{\eta TV}[u] = (I + \eta \partial TV)^{-1}(u) = \underset{ v \in L^2(\B)}{\operatorname{argmin}} \left\{ \frac{1}{2 \eta} \|u - v \|^2_{L^2(\B)} +  TV(v) \right\}.$$
is computed by using the dual approach introduced by Chambolle in \cite{chambolleTV}.

\bigskip

Finally, the gradient $\nabla A_1(\1e_{\Sigma},p_0)$ is computed via a time reversal imaging approach. Indeed,  $ \nabla A_1(\1e_{\Sigma},p_0)$ given by \eqref{nablaA1}
can also be expressed (using the superposition principle)  as $\nabla A_1(\1e_{\Sigma},p_0) = v_s(T,\cdot)$, with
 $$ \begin{cases}
 \partial^2_{tt} v_s - \Delta v_{s} = F,   \quad(t,x) \in \R^+\times \R^d,\\
 v_s(0,x) = 0 = \partial_t v_s(0,x) = 0, \quad x \in \R^d,
 \end{cases}
 $$
 where the right hand-side term is the measure 
$$
 F = \partial_t \left( \delta_{\{t=s\}} \left[ p_{[p_0]}(T-s,\cdot) - p_{obs}(T-s,\cdot) \right] \right) \1e_{\Sigma}. 
 $$

\subsection{Time reversal Imaging, wave equation and discretization} \label{sec:discretization_wave_equation}

Recall that the direct problem and the time reversal imaging approach only require to solve a Cauchy wave equation on the form
 $$ \begin{cases}
 \partial^2_{tt} v_s(t,x) - \Delta v_{s}(t,x) = F(t,x),  \quad(t,x) \in \R^+\times \R^d, \\
 v_s(0,x) = H_{1}(x), \quad \partial_t v_s(0,x) = H_{2}(x), \quad x \in \R^d,
 \end{cases}
 $$

\paragraph{Comments on numerics.} 
Regarding the numerical discretization, all the wave-like equations are solved in the box $\B=[-D/2,D/2]^2$ with periodic boundary conditions, where $D$ is supposed to be sufficiently large to prevent 
any reflection on the boundary. Numerical integrations of each equation are then performed exactly in the Fourier space. 

We recall that the Fourier truncation of a  two-dimensional function $u$ to the $M$ first modes, in a box $\B = [-D/2,D/2]^2$ is given by 
$$ 
u^{M}(t,x) = \sum_{n_1,n_2 = -[M/2]}^{[M/2]} c_{ {\bf n}}(t) e^{2i\pi {\bf \xi_n}\cdot x}
$$
where ${\bf n} = (n_1,n_2)$, ${\bf \xi_n} = (n_1/D,n_2/D)$ and $[\cdot]$ stands for the integer part function. Here, the coefficients $c_{ {\bf n}}$ stand for the $(2[M/2])^2$ first discrete Fourier coefficients of $u$.
Moreover, we use the inverse fast Fourier transform (denoted $IFFT$) to compute the inverse discrete Fourier transform of $c_{\bf n}$. This leads to $u^{M}_{{\bf n}} = IFFT[c_{\bf n}]$ where $u^{M}_{{\bf n}}$ 
is the value of $u$ at the points $x_{{\bf n}} = (n_1 h,n_2 h)$ where $h =D/M$. \\
Conversely, $c_{ {\bf n}}$ can be computed by applying the discrete Fourier transform to $u^M_{{\bf n}}$: 
$$ c_{ {\bf n}} = FFT[u^M_{{\bf n}}].$$

Now, about computation of wave equation solutions, remark that a generic wave equation 
$$ \partial^2_{tt} u^M(t,x) - \Delta u^M(t,x) = F^M(t,x) = \sum_{n_1,n_2 = -M/2}^{M} f_{ {\bf n}}(t) e^{2i\pi {\bf \xi_n}\cdot x}$$
reads
$$
\partial_t \begin{pmatrix}
 u^M \\
 u^M_t
\end{pmatrix} = 
\begin{pmatrix}
 0& I_d \\
 \Delta& 0 
\end{pmatrix}\begin{pmatrix}
 u^M \\
 u^M_t
\end{pmatrix} + \begin{pmatrix}
 0 \\
 F^M
\end{pmatrix}
$$
which can by simply integrated by solving the linear EDO system
$$
\frac{d}{dt}\begin{pmatrix}
 c_{{\bf n}}(t) \\
 c_{{\bf n}}'(t) 
\end{pmatrix} = 
\begin{pmatrix}
 0& 1 \\
 -4 \pi^2 |\xi_{{\bf n}}|^2 & 0 
\end{pmatrix}\begin{pmatrix}
 c_{{\bf n}}(t) \\
 c_{{\bf n}}'(t) 
\end{pmatrix} + \begin{pmatrix}
 0 \\
 f_{{\bf n}}(t) 
\end{pmatrix}.
$$

\subsection{Numerical experiments} \label{sec:numerical_experiment} 

All the numerical simulations of this section are done with the following set of parameters: 
\begin{itemize}
 \item the set $\Omega$ is a two-dimensional ball of radius $1$
 \item the box $\B = [-D/2,D/2]^d$ has size $D=4$ and the record time is $T=2$;
 \item the set $K$ is a two-dimensional ball of radius $0.85$
 \item we use a regular step discretization with parameters $dt = T/2^{10}$ and $dx = D/2^{9}$.
 \item the thickness parameter $\varepsilon$ is equal to $0.03$. 
 \item the $TV$-parameter $\gamma = 0.01$. 
 \item the descent step $\eta = 0.5$.
 \end{itemize}
 
On Figure \ref{fig:time_reversal_delta}, we provide a first experiment using ideal complete data $ g = p_{obs}$ 
on the whole boundary $\partial\Omega$ as well as the time reversal imaging 
$\I(g)$ for three different values of the pressure $p_0$. It is observed that in each case, the reconstructed source and the exact source are very close. 
\begin{figure}[H]
\begin{center}
	\includegraphics[width=0.3\linewidth]{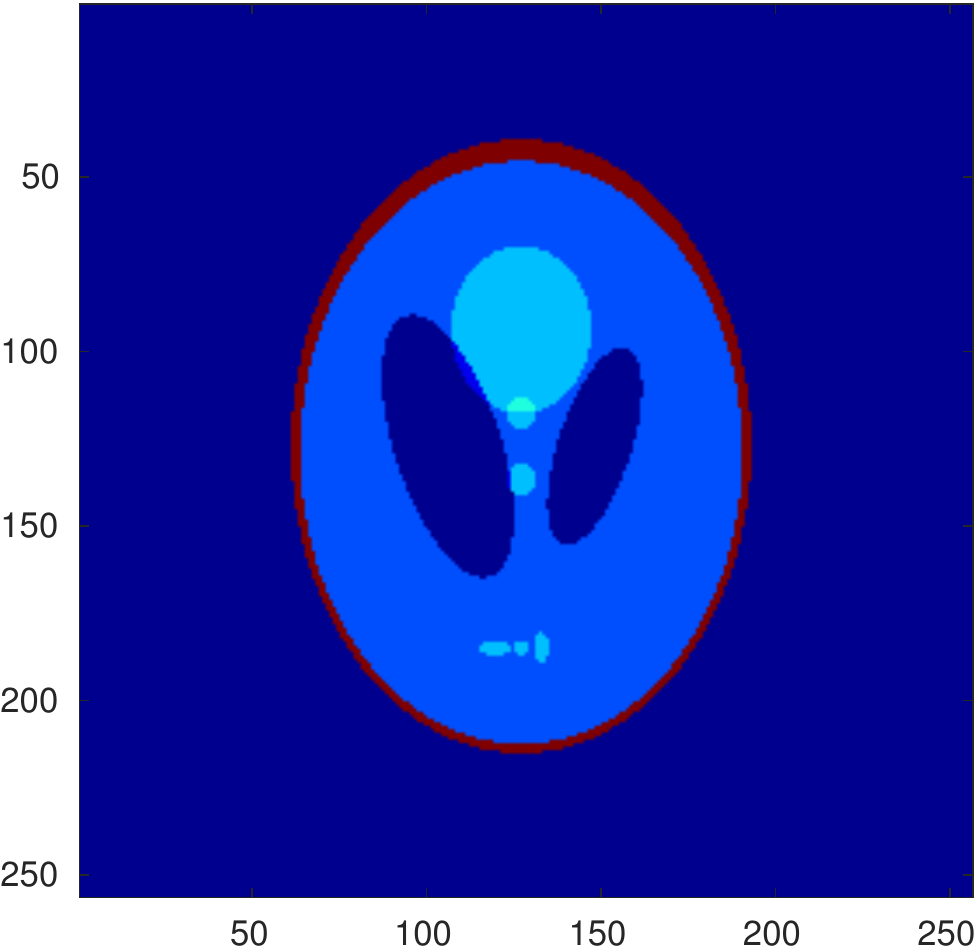}
	\includegraphics[width=0.37\linewidth]{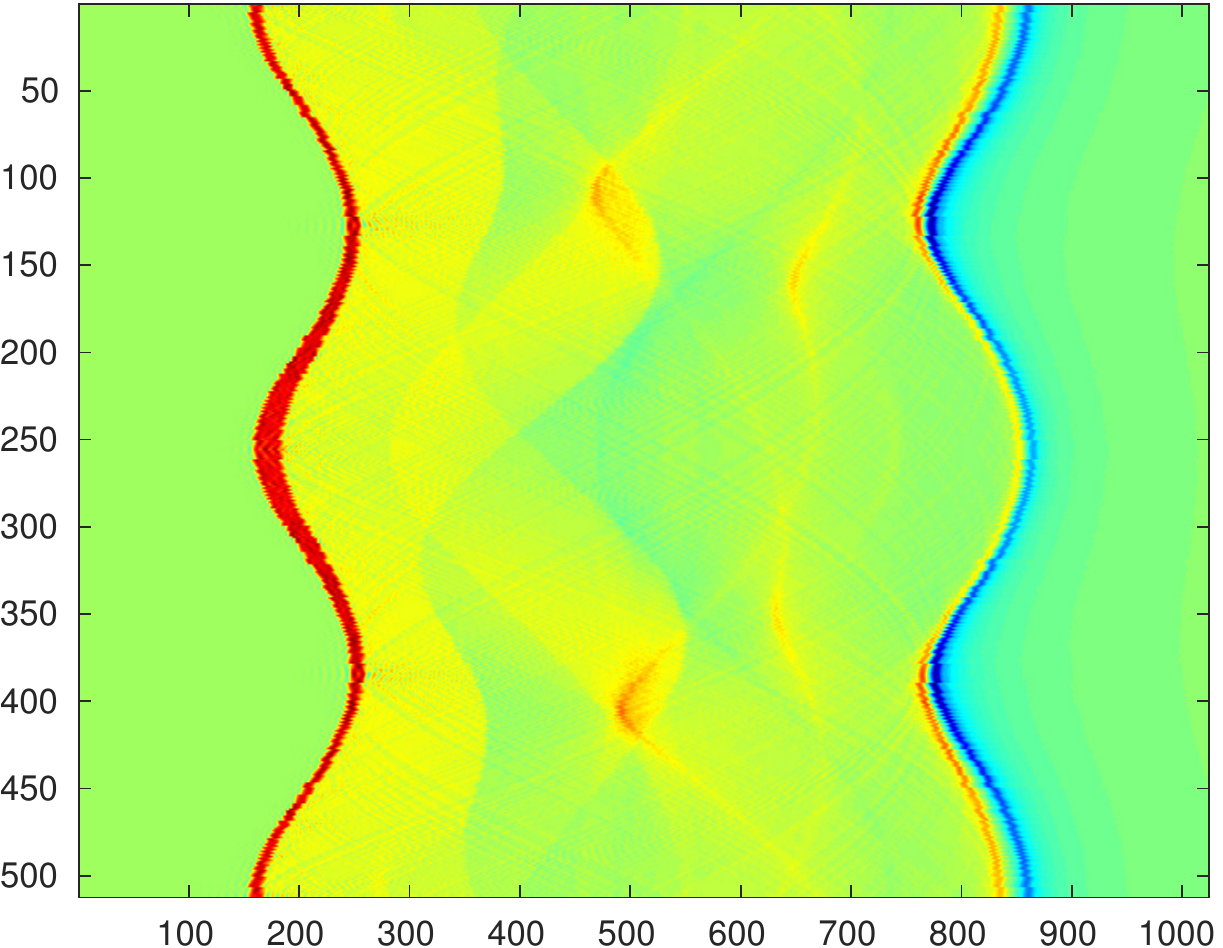} 
	\includegraphics[width=0.3\linewidth]{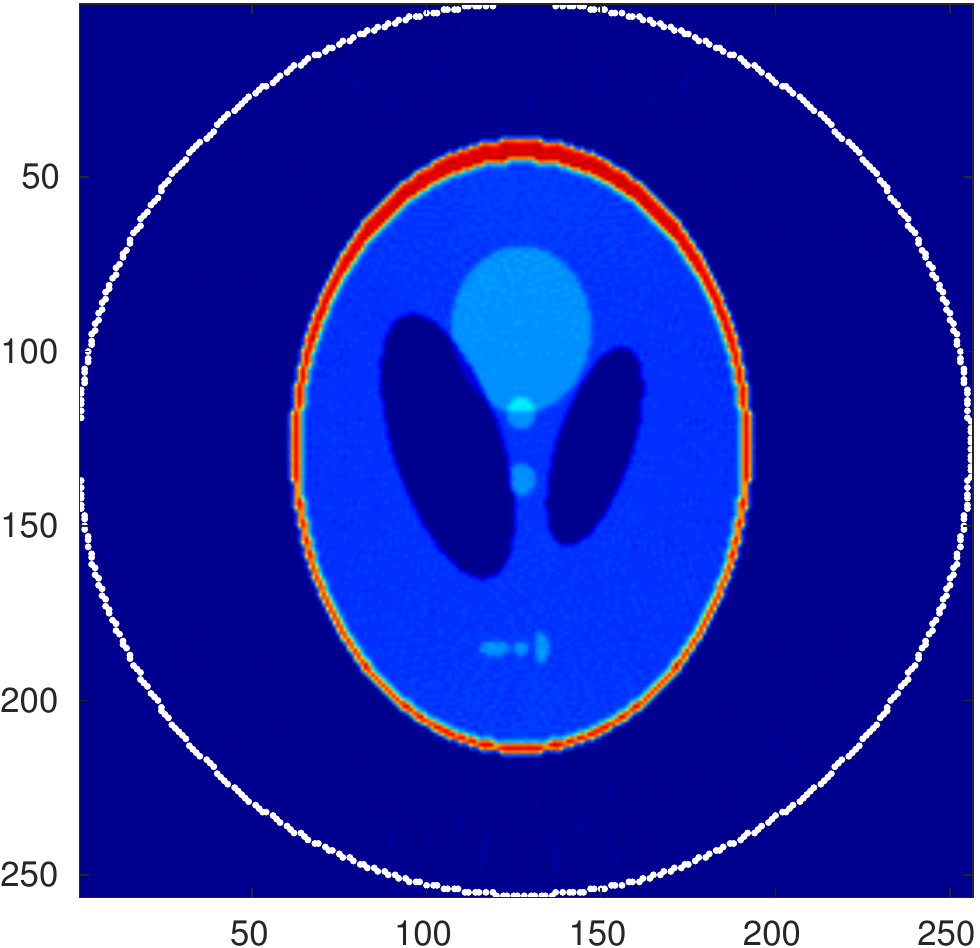} \\
	\includegraphics[width=0.3\linewidth]{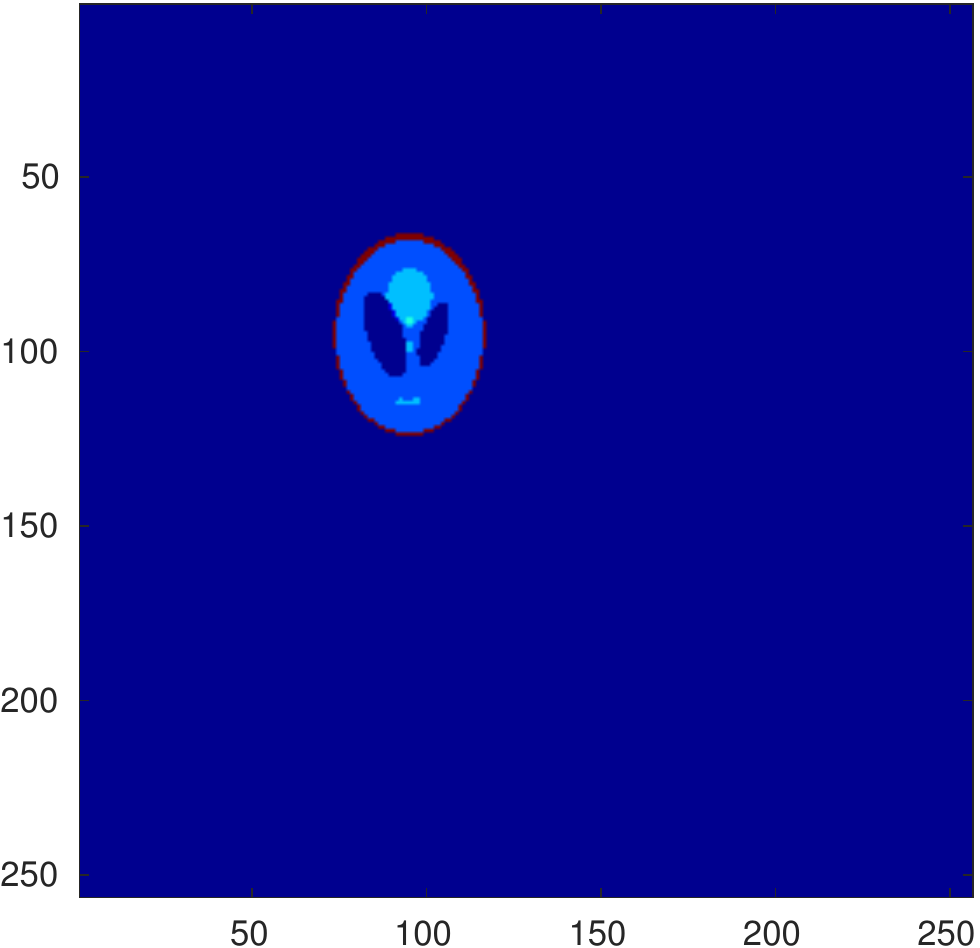}
	\includegraphics[width=0.37\linewidth]{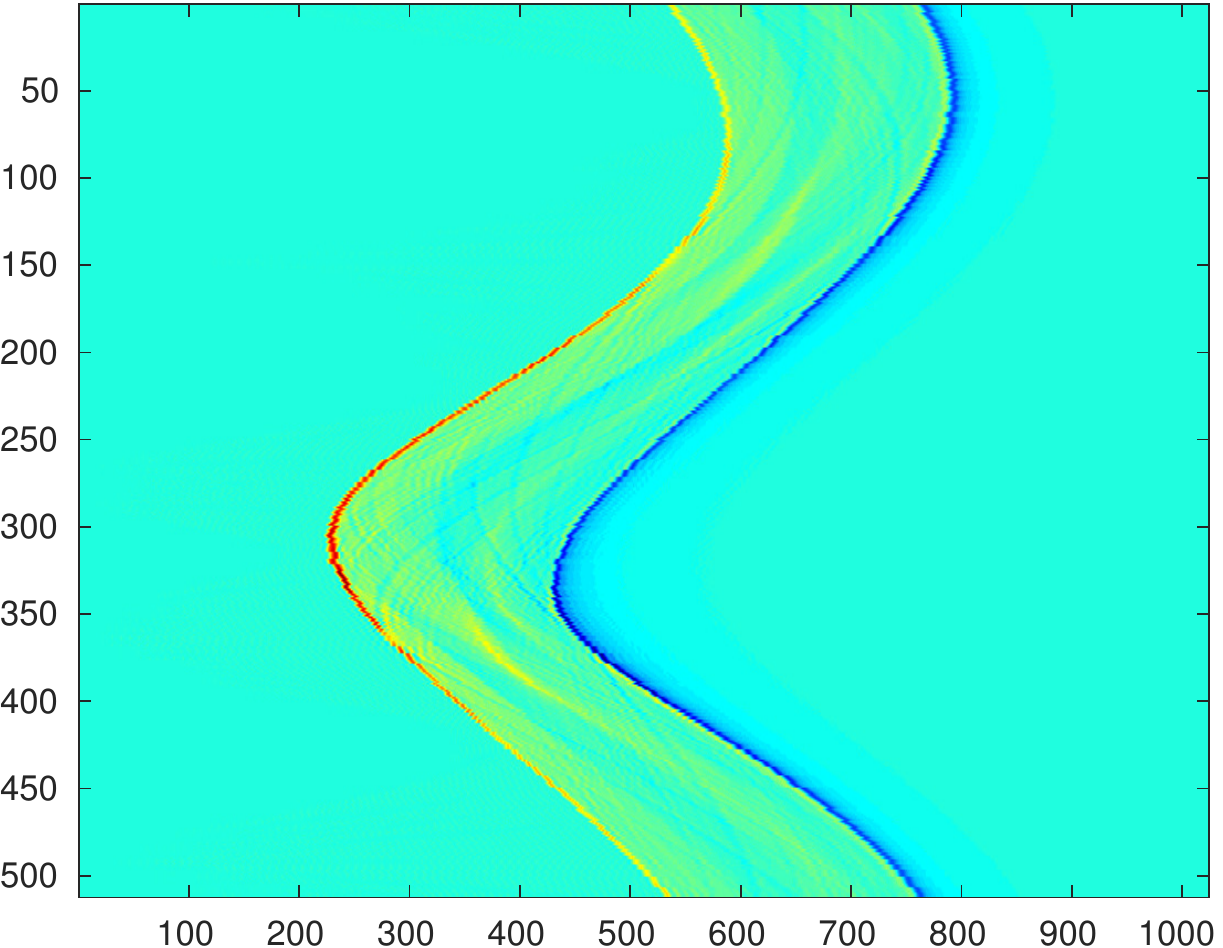} 
	\includegraphics[width=0.3\linewidth]{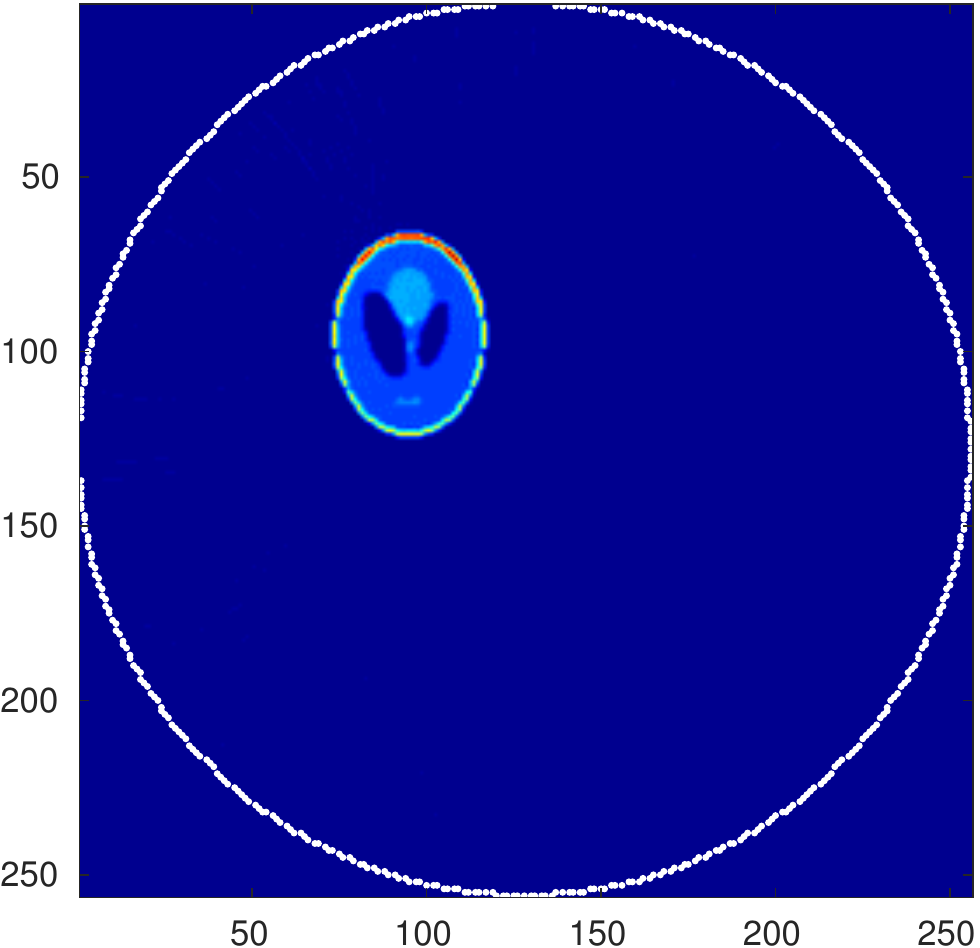} \\
	\includegraphics[width=0.3\linewidth]{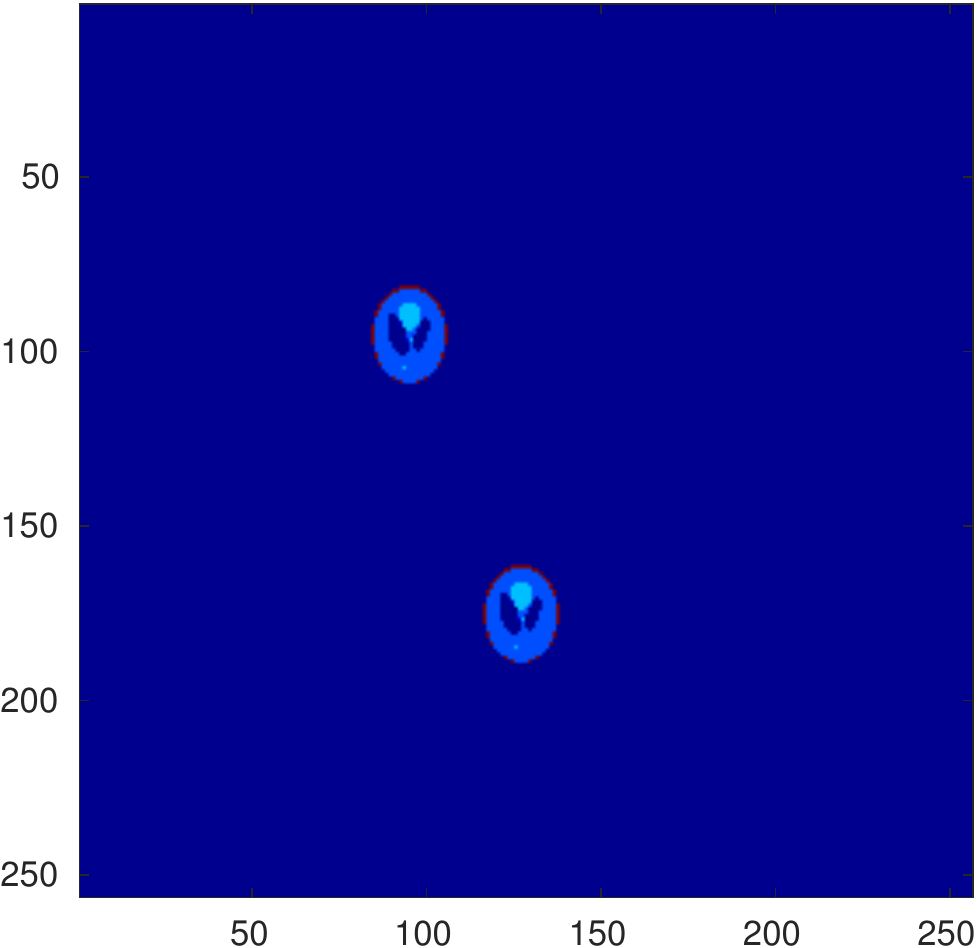}
	\includegraphics[width=0.38\linewidth]{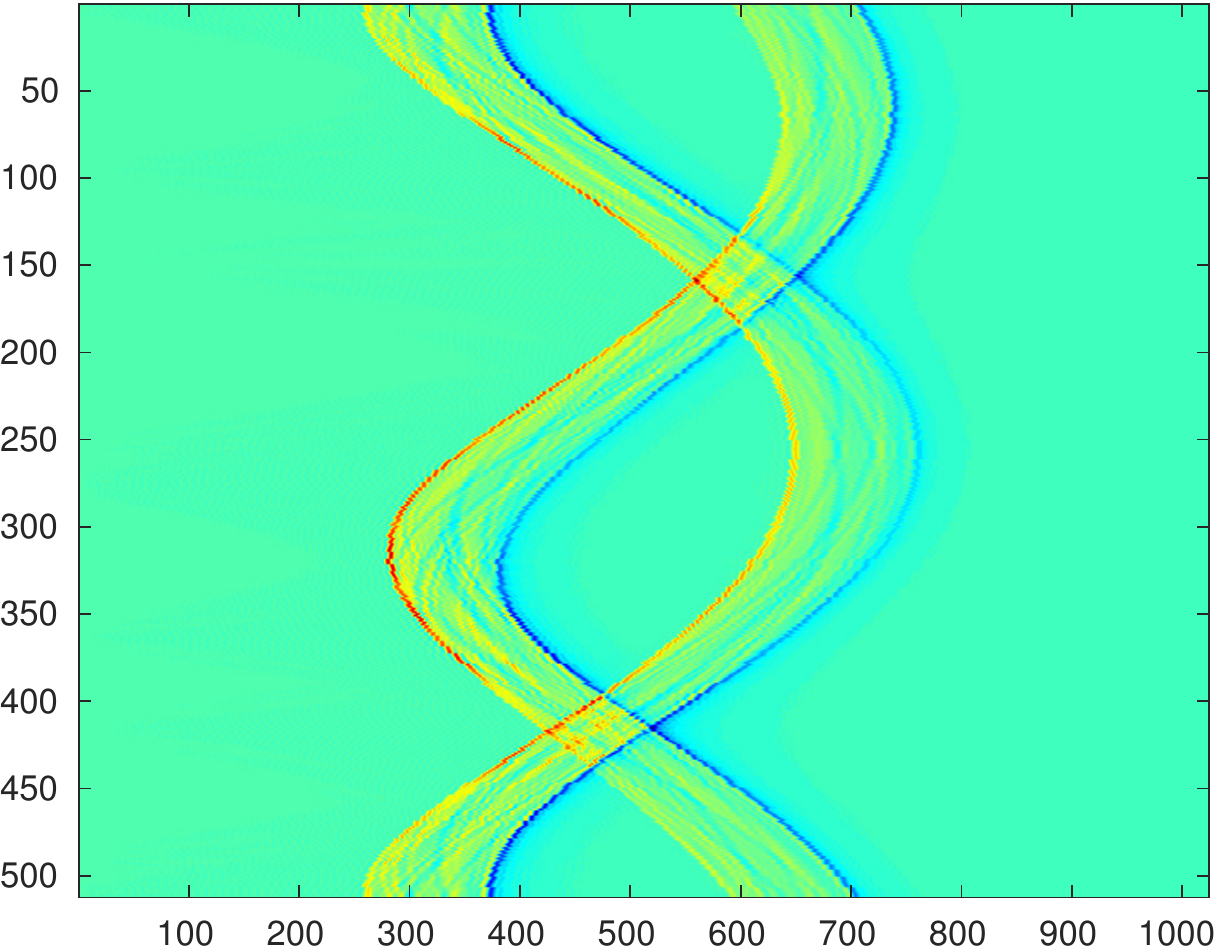} 
	\includegraphics[width=0.3\linewidth]{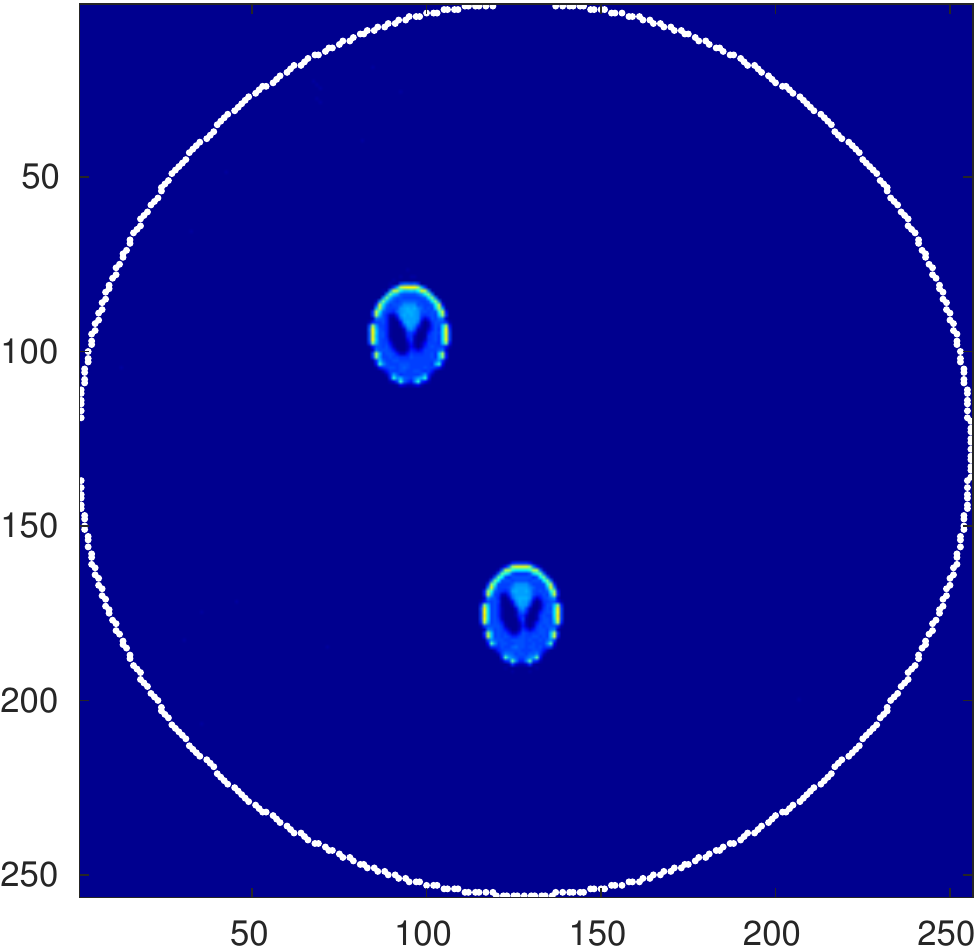} \\
\end{center}
\caption{Source reconstruction using time reversal imaging $\I$. Each line corresponds to a different choice of the source $p_0$. Left: initial source $p_0$;
middle: data $g = p_{obs}$; right: $\I[g]$. The white dots corresponds to sensors locations. 
 \label{fig:time_reversal_delta}}
\end{figure}  
On Figure \ref{fig:time_reversal_iterationTV1}, we give the result of the reconstruction procedure described in Section \ref{sec:criteriachoices}   using partial data $ g = p_{obs}$ with $L = 0.3$. The location of sensors are plotted with white marks on each picture. Each line corresponds to a different choice of the source $p_0$:  we plotted the source $p_0$, the result of the reconstruction by using the time reversal imaging $p_0^0$
and the reconstruction of the source $p_0^n$ after $n=30$ iterations.
In particular, we observed that some information on $p_0$ are lost and in particular the discontinuities
with normal directions that do not meet any sensor. \\

\begin{figure}[h!] 
 \begin{center}
 	\includegraphics[width=0.32\linewidth]{TestTR_Nsource_1_source.pdf}
 	\includegraphics[width=0.32\linewidth]{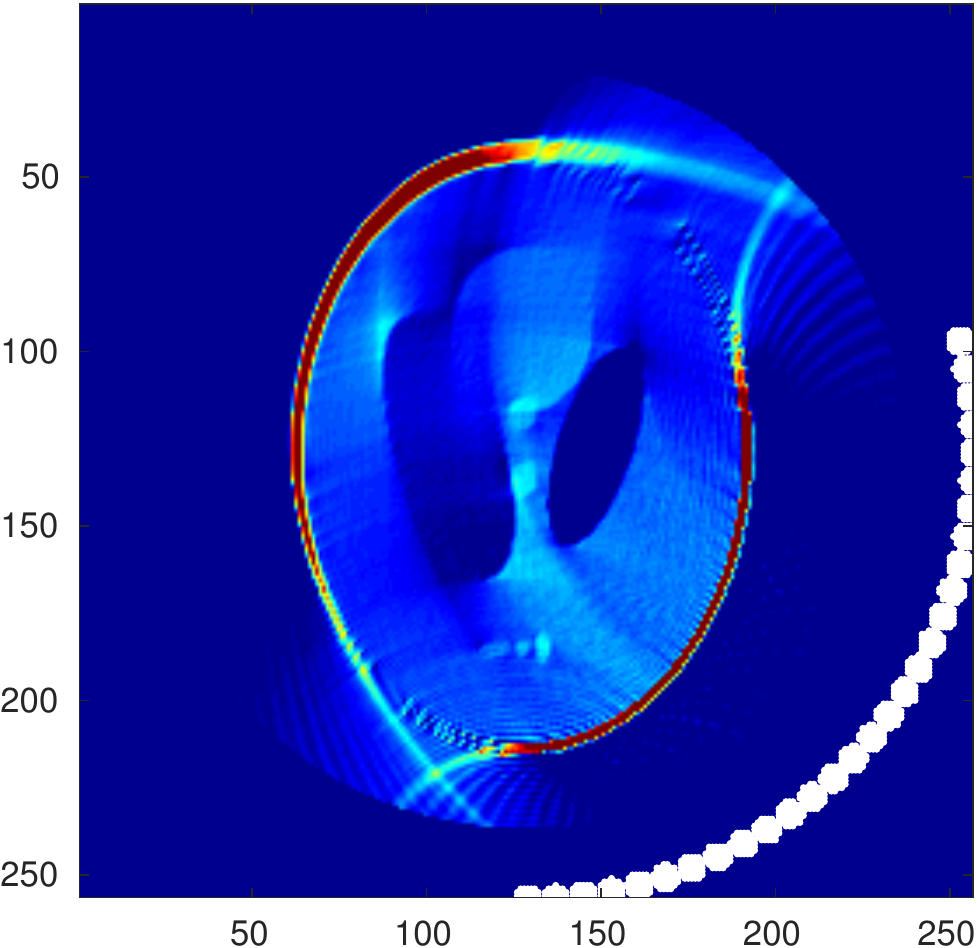} 
 	\includegraphics[width=0.32\linewidth]{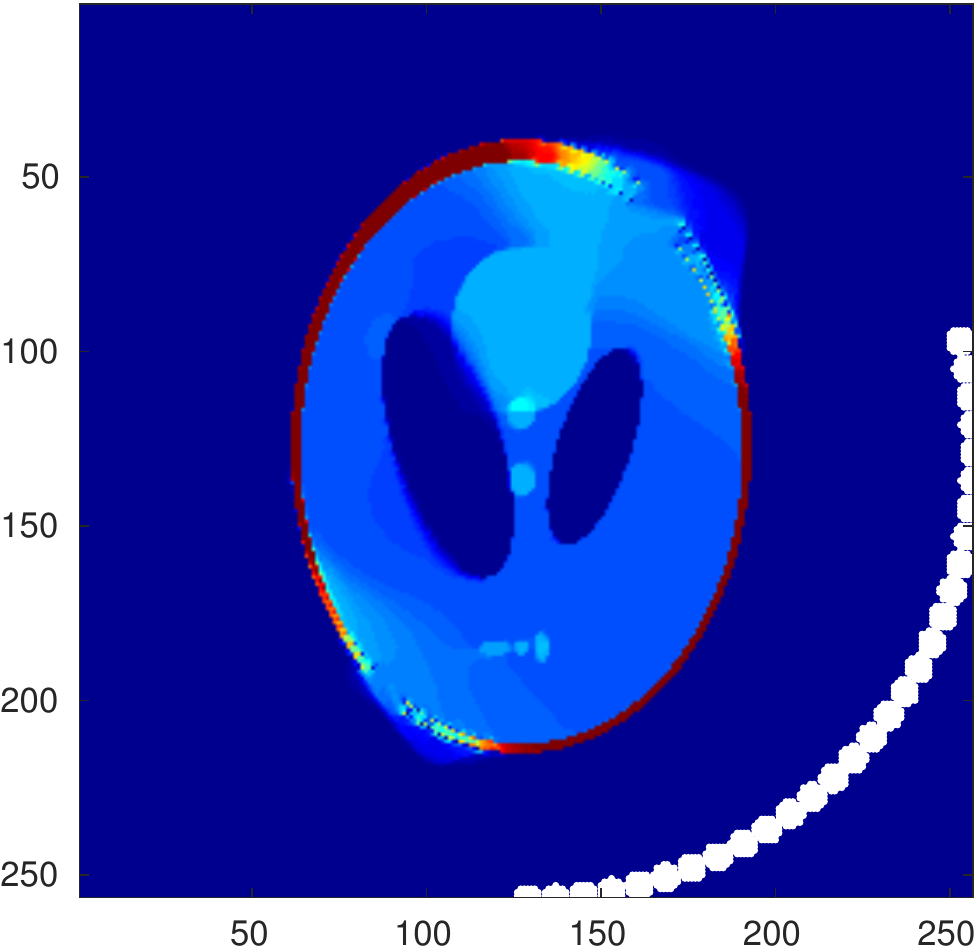} \\
 	\includegraphics[width=0.32\linewidth]{TestTR_Nsource_2_source.pdf}
 	\includegraphics[width=0.32\linewidth]{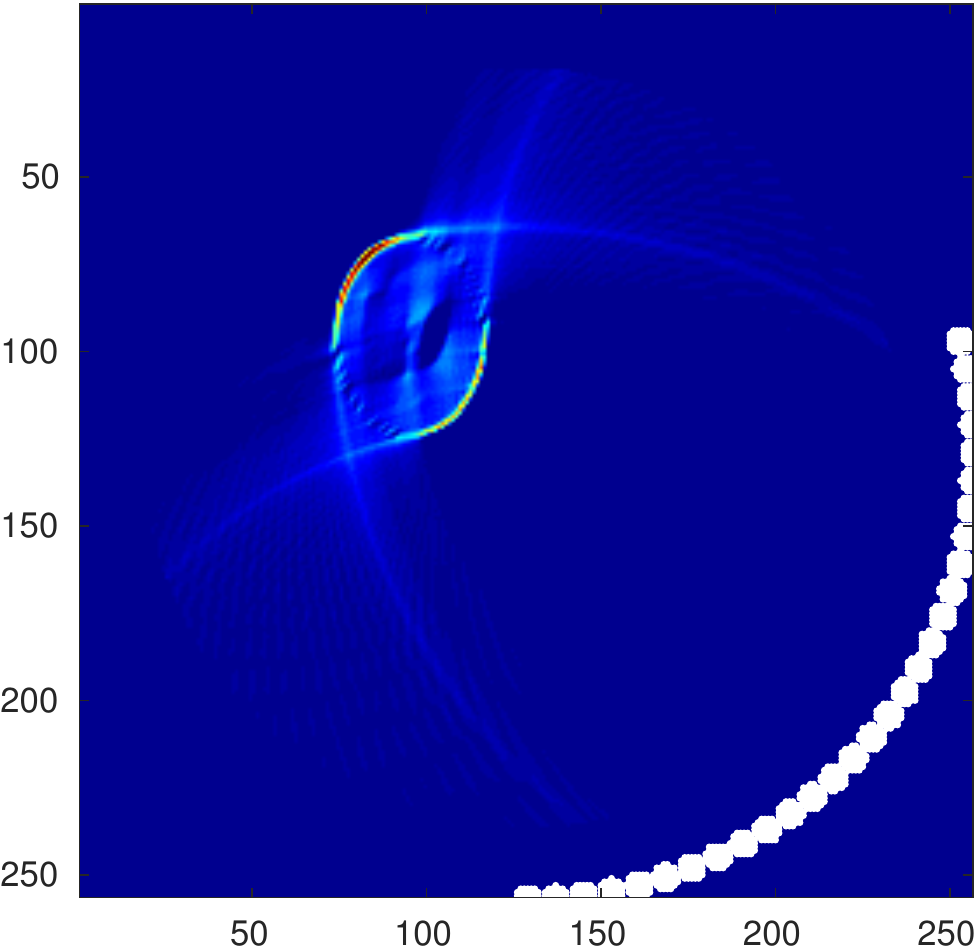} 
 	\includegraphics[width=0.32\linewidth]{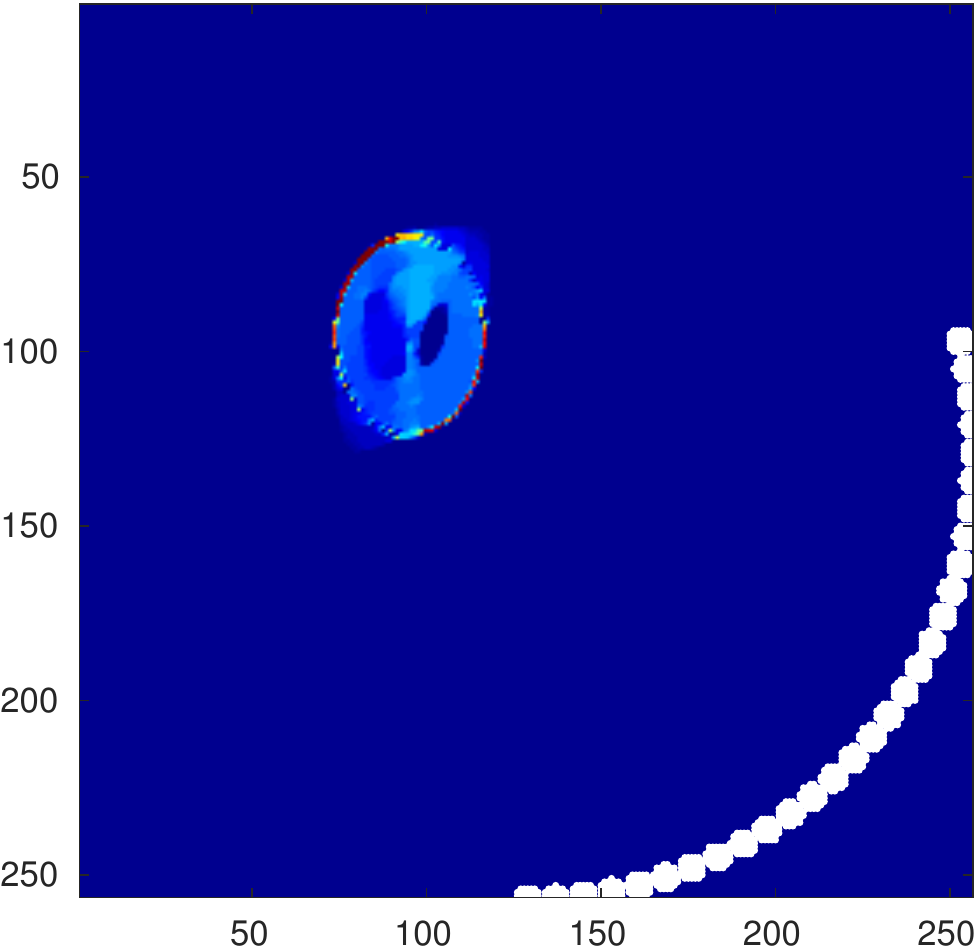} \\
 	\includegraphics[width=0.32\linewidth]{TestTR_Nsource_3_source.pdf}
 	\includegraphics[width=0.32\linewidth]{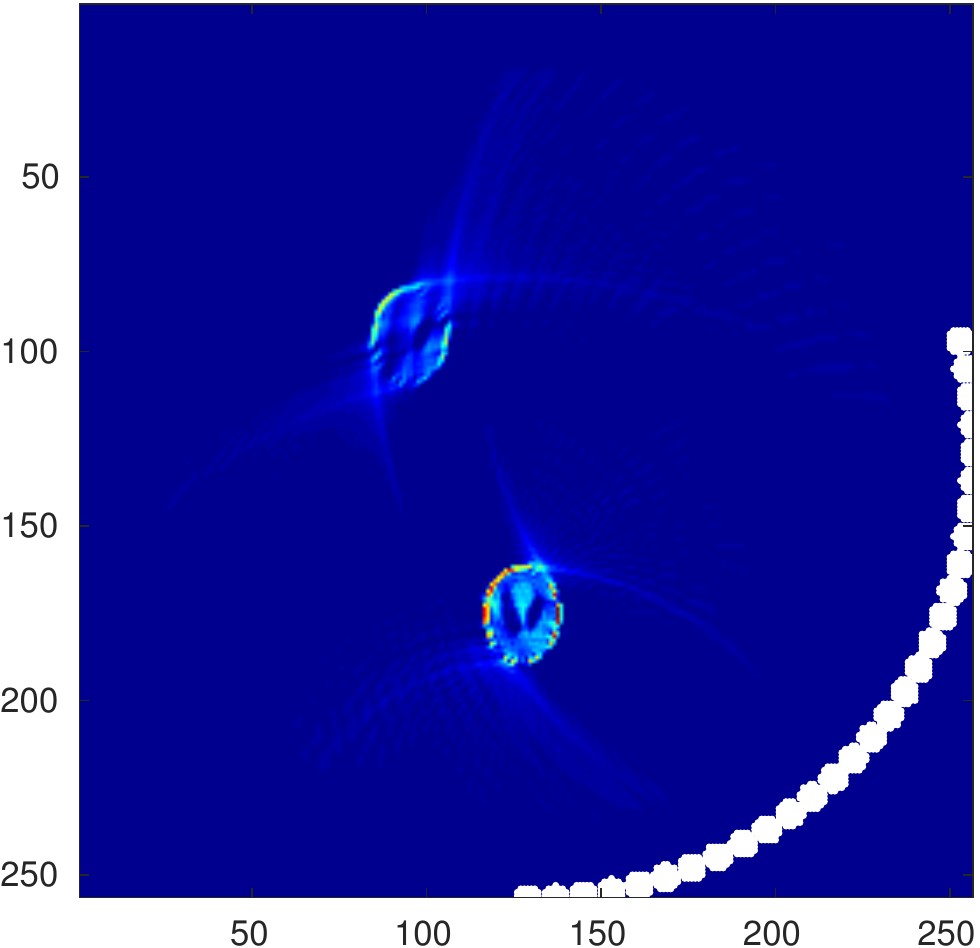} 
 	\includegraphics[width=0.32\linewidth]{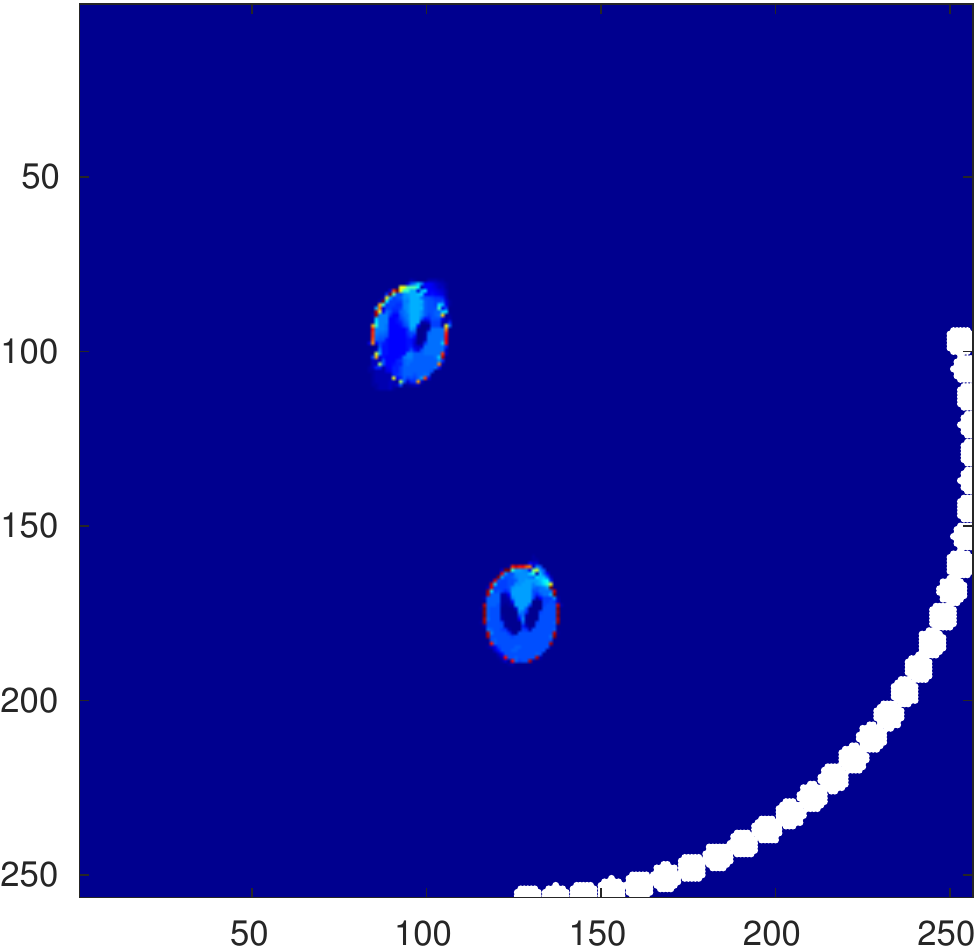} \\
\end{center}
 \caption{Minimization of $J^0$ w.r.t. $(\1e_{\Sigma},p_0)$.  Each line corresponds to a different choice of the source $p_0$. Left: initial source $p_0$; middle: reconstruction using time reversal imaging- $\I[g]$; right: reconstructed source $p^{n}_{0}$ after $n= 30$ iterations. 
 The white dots correspond to sensors locations. \label{fig:time_reversal_iterationTV1}}
 \end{figure} 
Next, on Figure \ref{fig:time_reversal_iterationTV}, we present the reconstruction of the source $p_0$ if one allows the sensors location to evolve. As previously, each line corresponds to a different choice of the source $p_0$.
Moreover, on each line, we respectively plotted
\begin{itemize}
 \item the reconstructed source $p^{n}_{0}$ after $n= 30$ iterations, 
 \item the energy function $\psi_{u[p^{n}_{0}]}$ computed on $ \Omega$ and the associated optimal location of sensors plotted with red marks. 
 \item the reconstructed source $p^{n}_{0}$ after $n= 15$ iterations after using the new location of sensors. 
\end{itemize}

\begin{figure}[H] 
 \begin{center}
 	\includegraphics[width=0.32\linewidth] {Test2_Nsource1_k_100.pdf}
 	\includegraphics[width=0.32\linewidth]{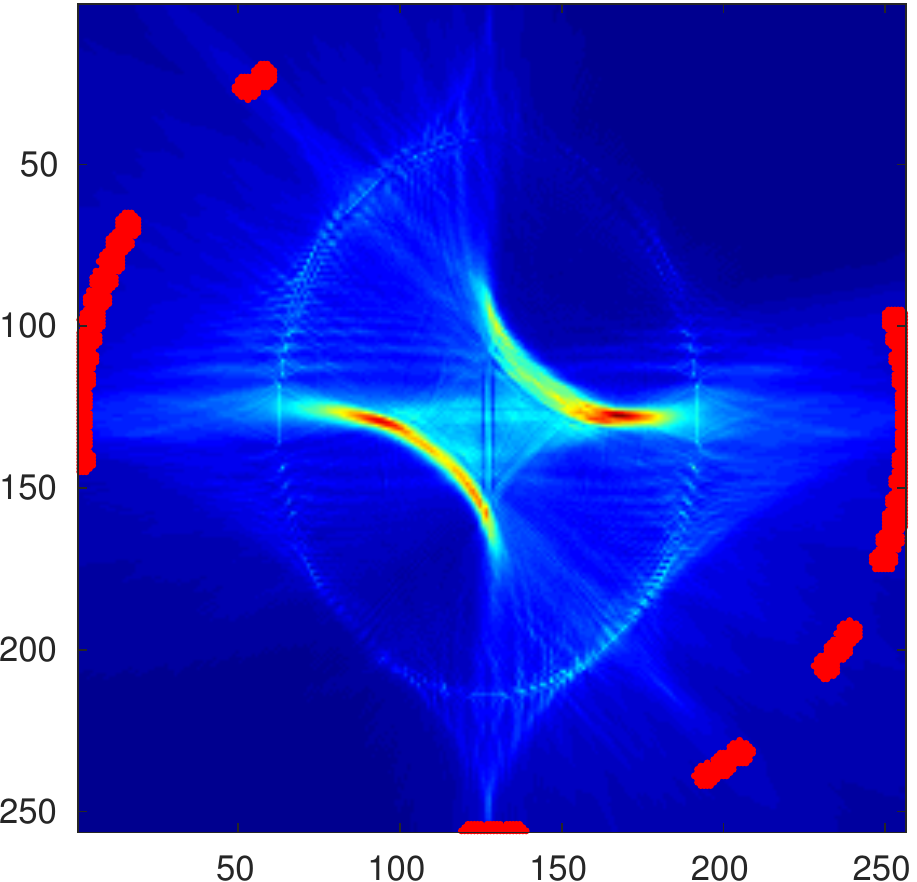} 
 	\includegraphics[width=0.32\linewidth]{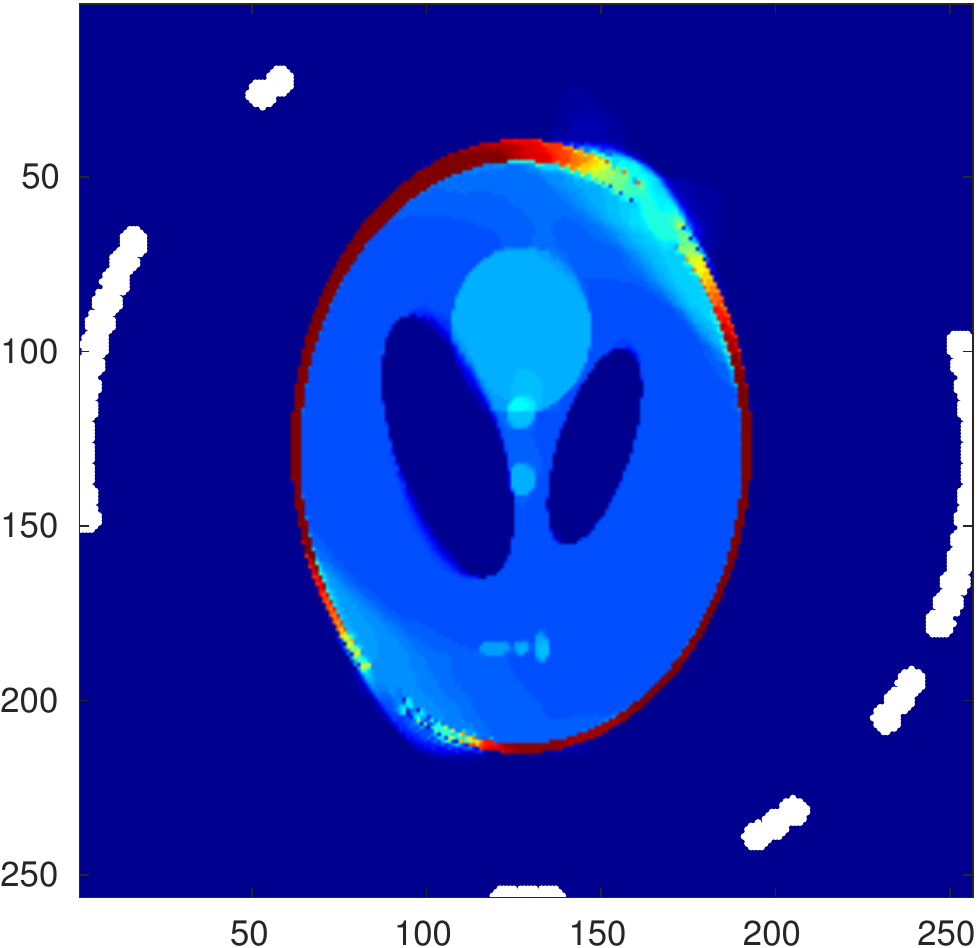} \\
 \includegraphics[width=0.32\linewidth]{Test_Nsource2_k_30.pdf}
 	\includegraphics[width=0.32\linewidth]{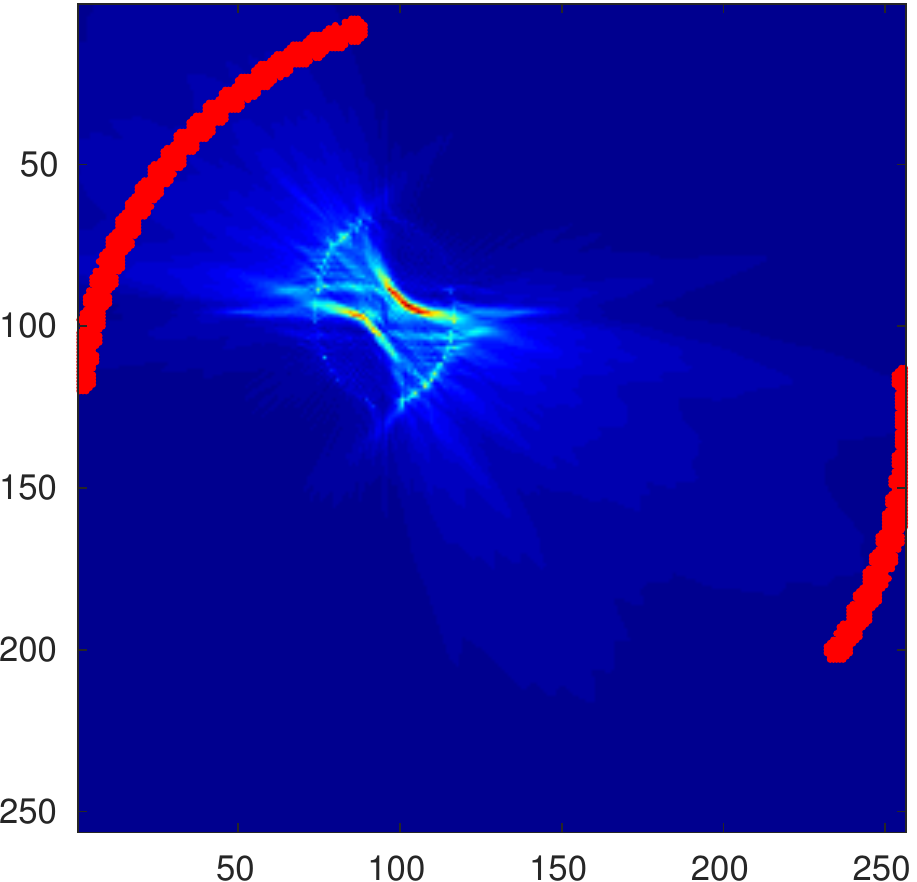} 
 	\includegraphics[width=0.32\linewidth]{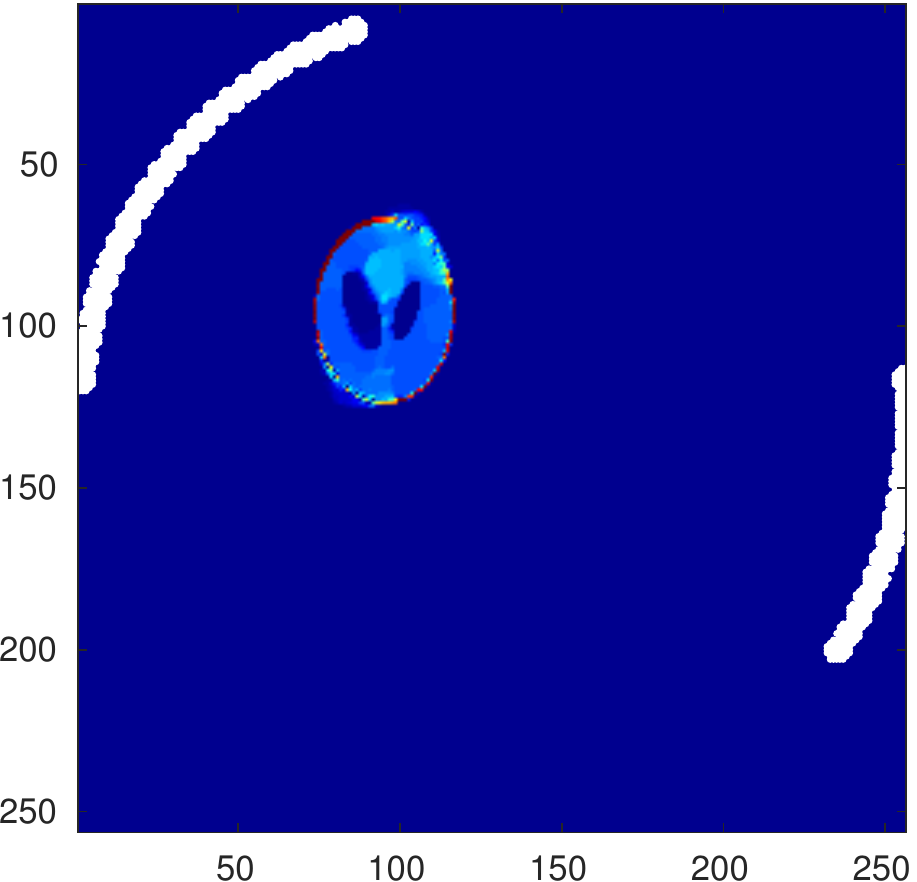} \\
 \includegraphics[width=0.32\linewidth]{Test_Nsource3_k_30.pdf}
 	\includegraphics[width=0.32\linewidth]{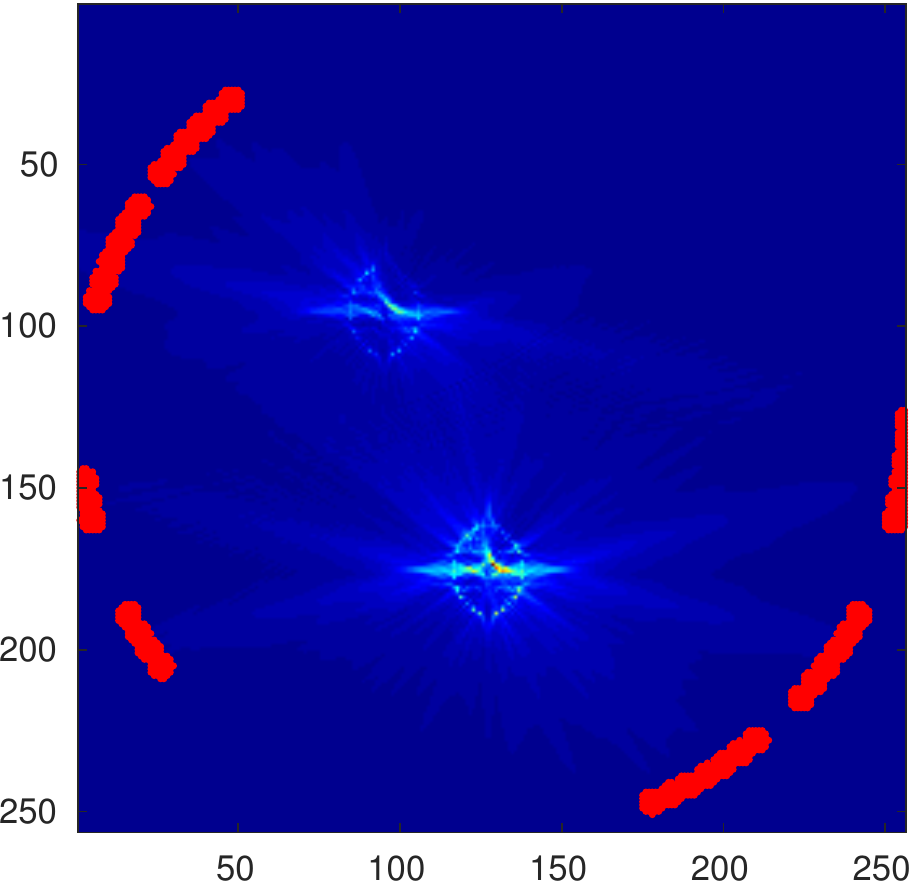} 
 	\includegraphics[width=0.32\linewidth]{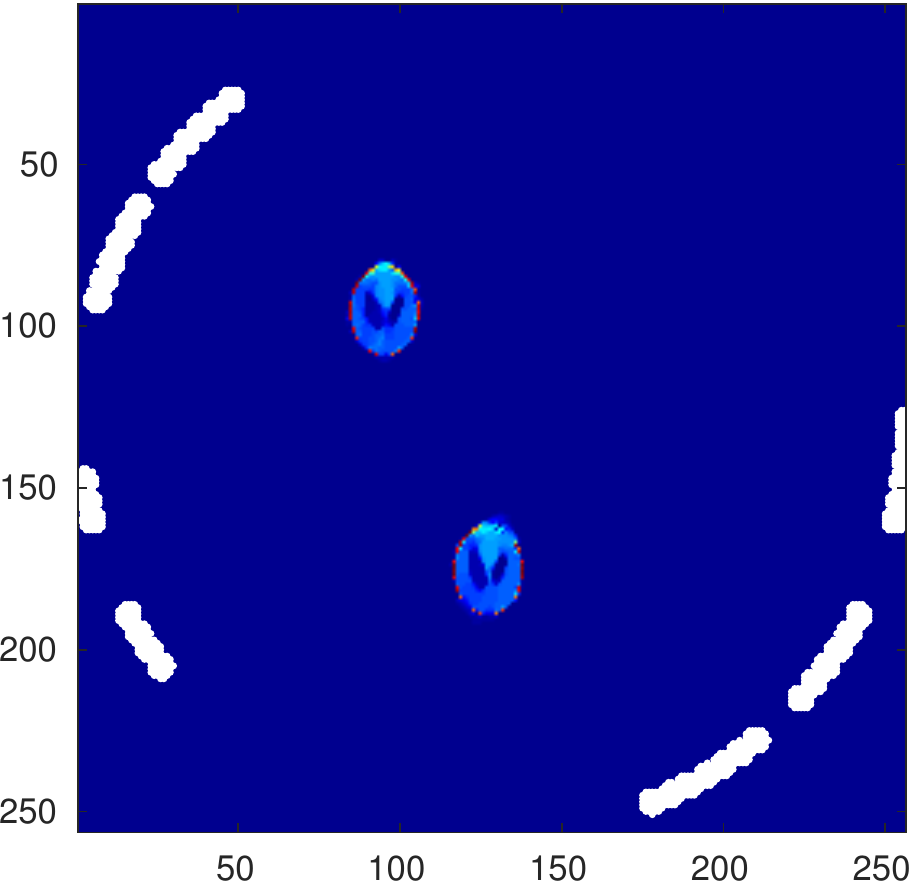} \\
\end{center}
 \caption{Optimization of sensors location. Each line corresponds to a different choice of the source $p_0$; 
 Left: reconstructed source $p^{n}_{0}$ after $n= 30$ iterations by using initial locations of the sensors;
 middle: function $ \psi_{p_0^n}$ defined by \eqref{n2202} on $ \Omega$ and new location of sensors ({\it red} dots); 
 right: reconstructed source $p^{n}_{0}$ after $n= 20$ iterations by using the new location of sensors.
 \label{fig:time_reversal_iterationTV}
 \label{fig:perspective2}}
 \end{figure} 
Finally, as expected, the reconstruction of the source is much better by using the new location of sensors even if, the reconstruction remains unperfect.

Let us illustrate the interest of using the term $A_2$ as a good reconstruction quality factor. On Figures \ref{fig:time_reversal_iterationTV2} and \ref{fig:time_reversal_iterationTV3}, we first compute the optimal location of sensors, respectively in the continuous and discrete\footnote{Meaning that we consider a given number of sensors.} settings, using the true value of the source term $p_0$. Second, using the new sensors location, we provide an estimate of the source $p_0$ by solving Problem \eqref{MainPbOpt}. 
In each case, we observe that the source reconstruction is almost perfect.

\begin{figure}[h!] 
 \begin{center}
 	\includegraphics[width=0.32\linewidth]{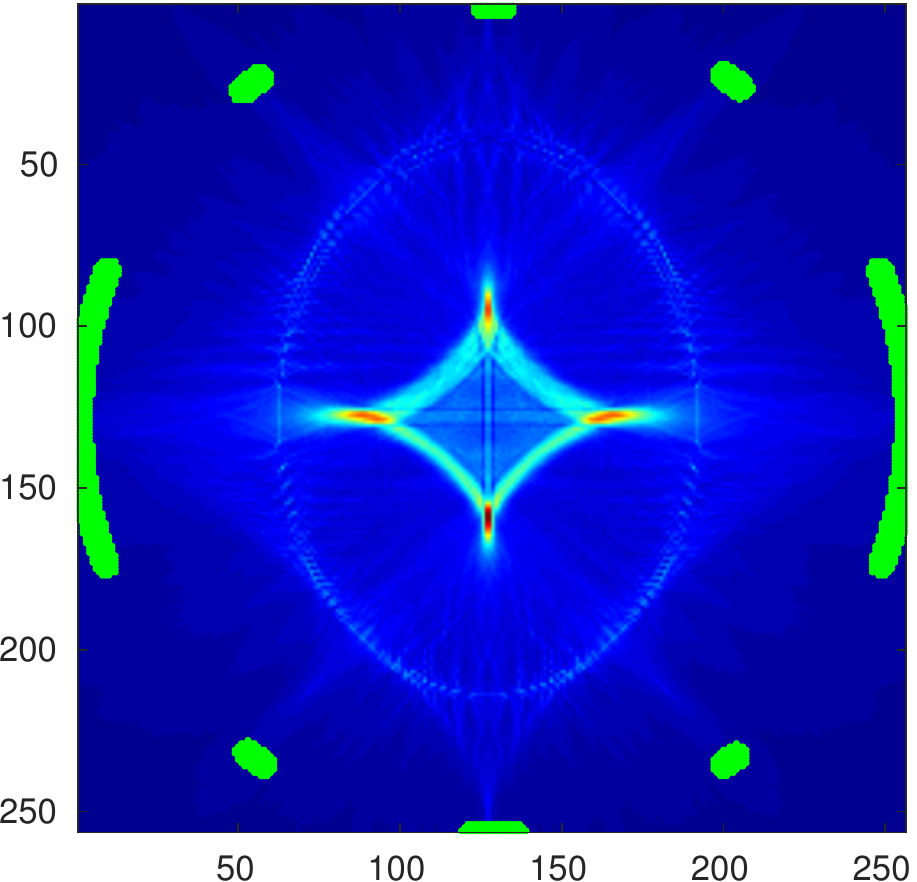}
 	\includegraphics[width=0.32\linewidth]{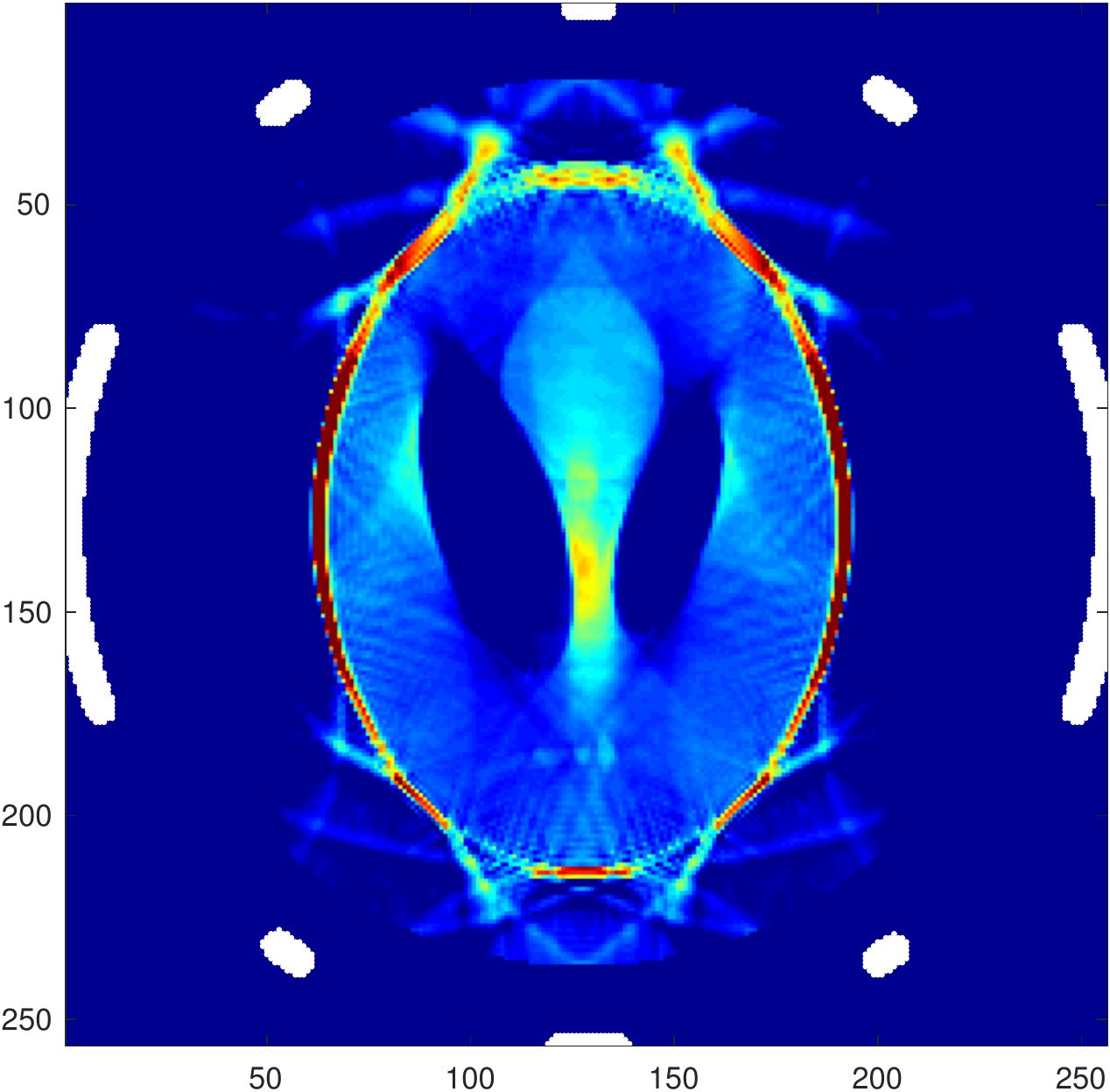} 
 	\includegraphics[width=0.32\linewidth]{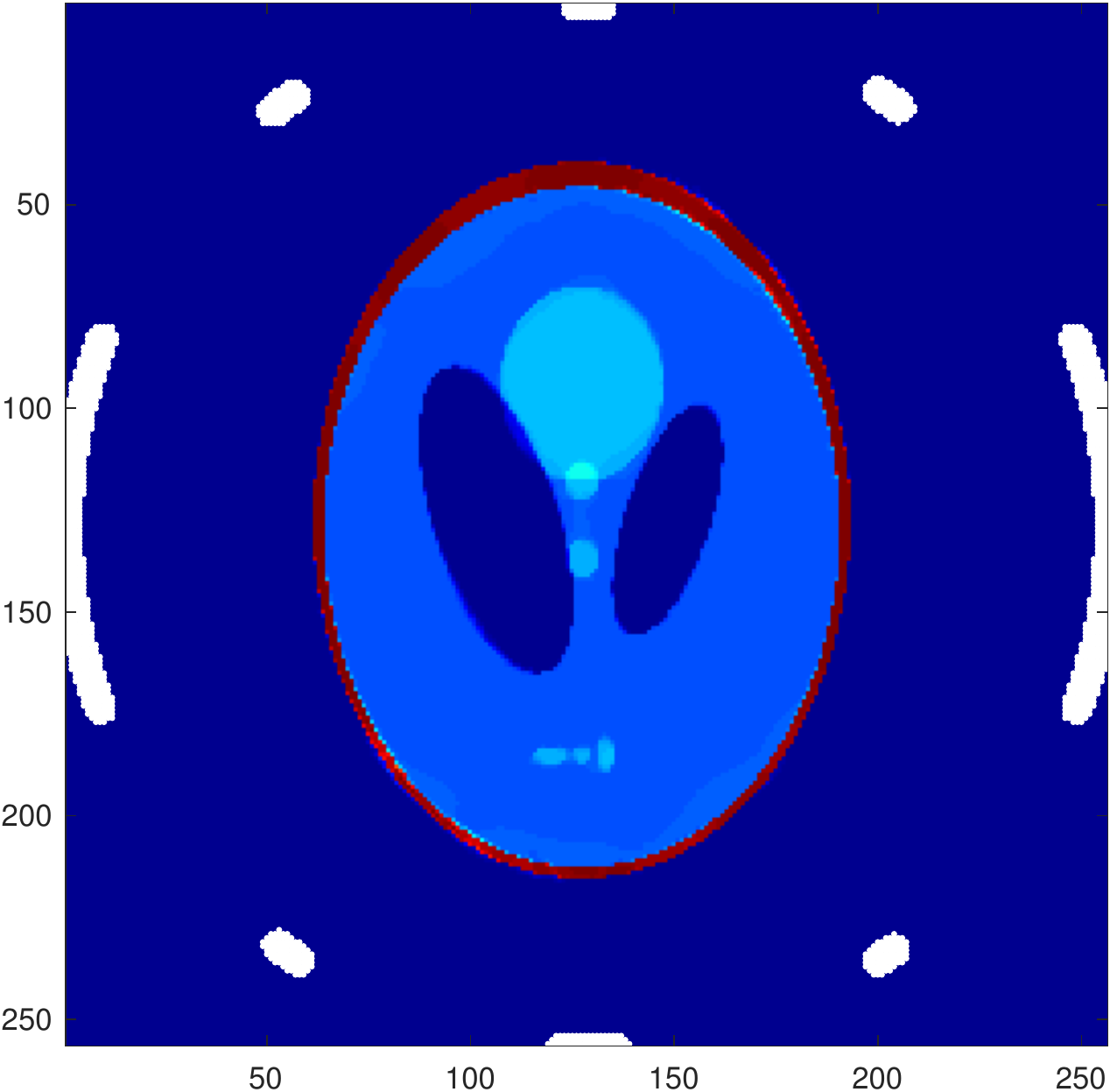} \\
 \includegraphics[width=0.32\linewidth]{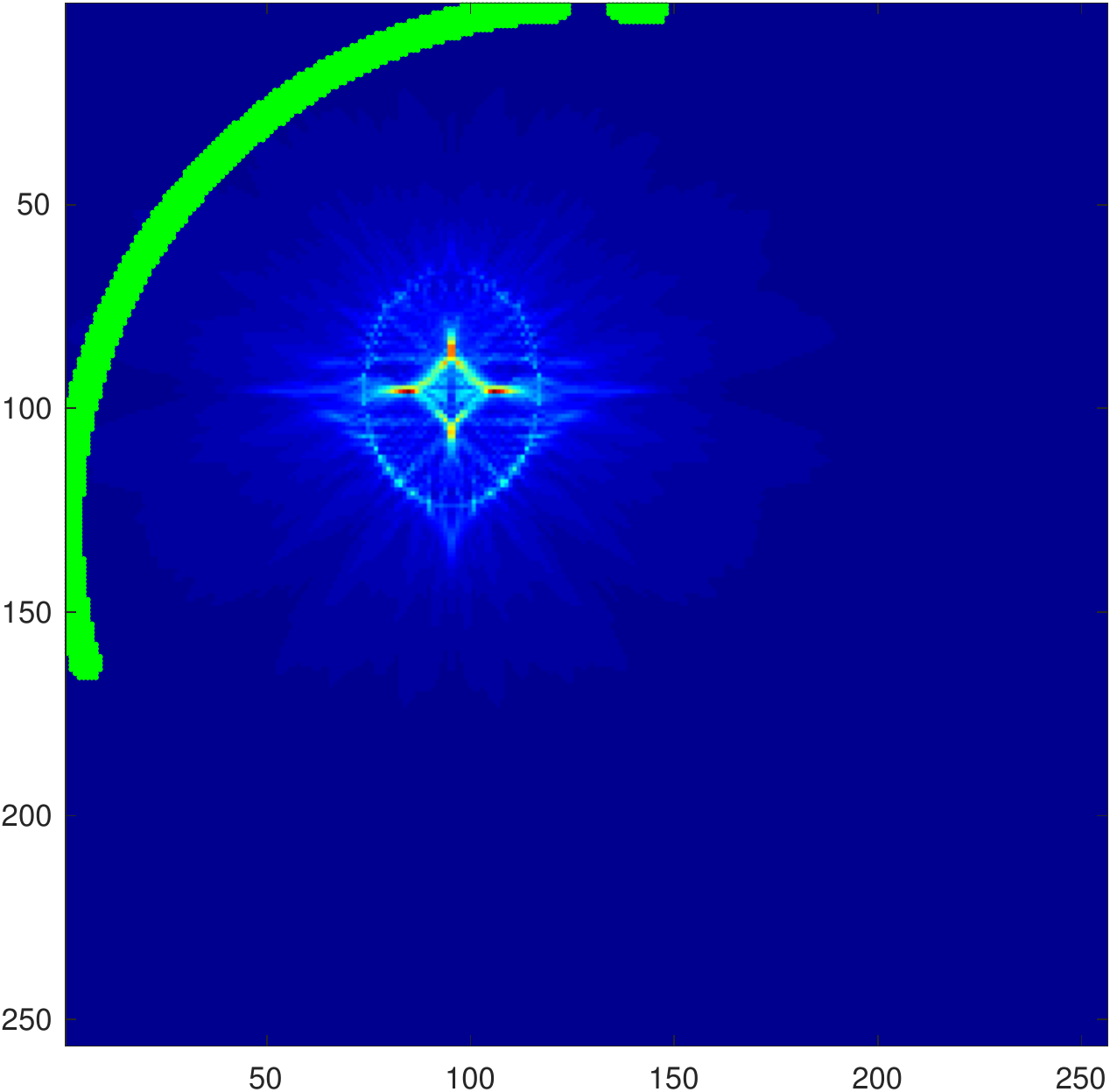}
 	\includegraphics[width=0.32\linewidth]{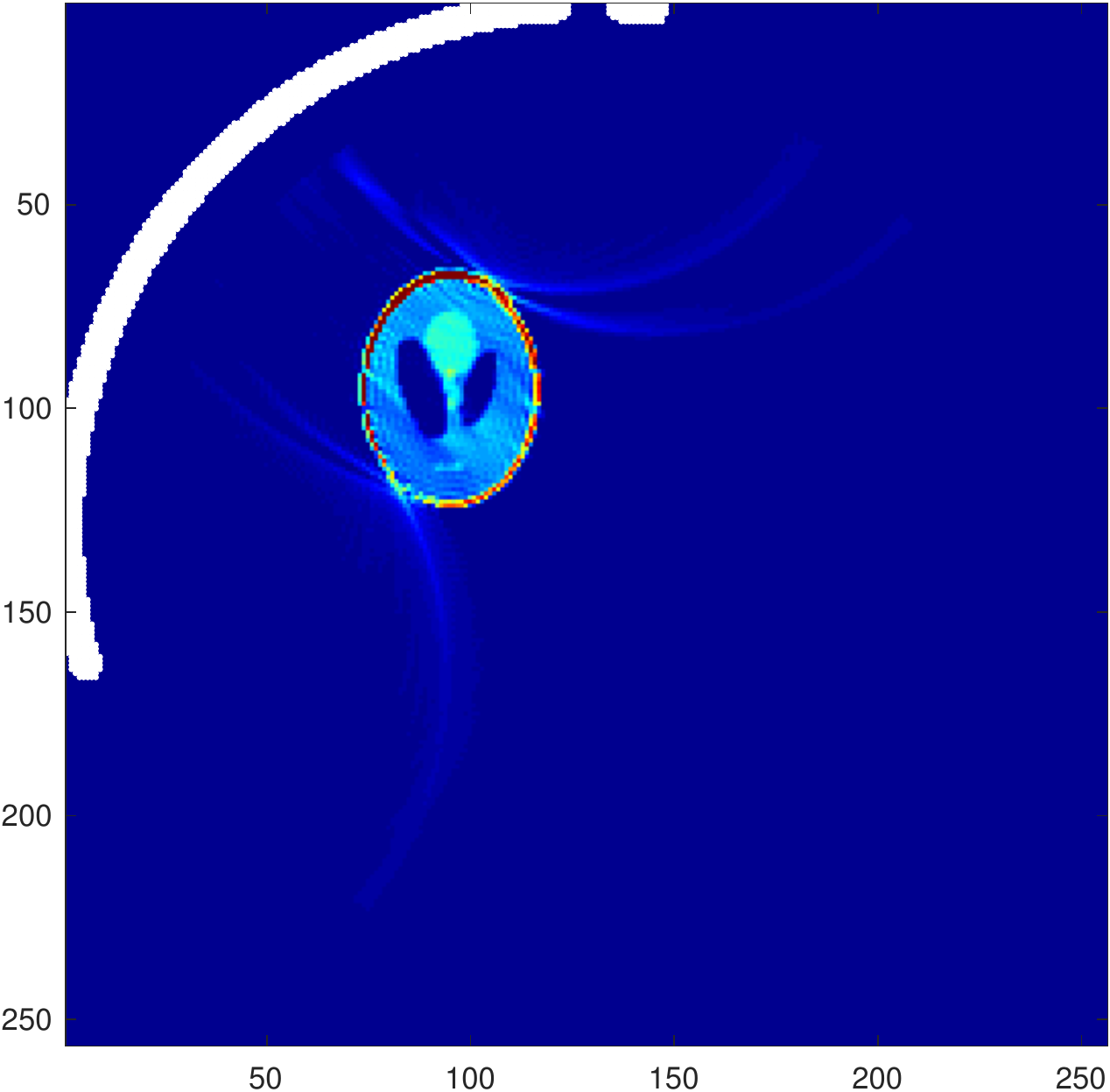} 
 	\includegraphics[width=0.32\linewidth]{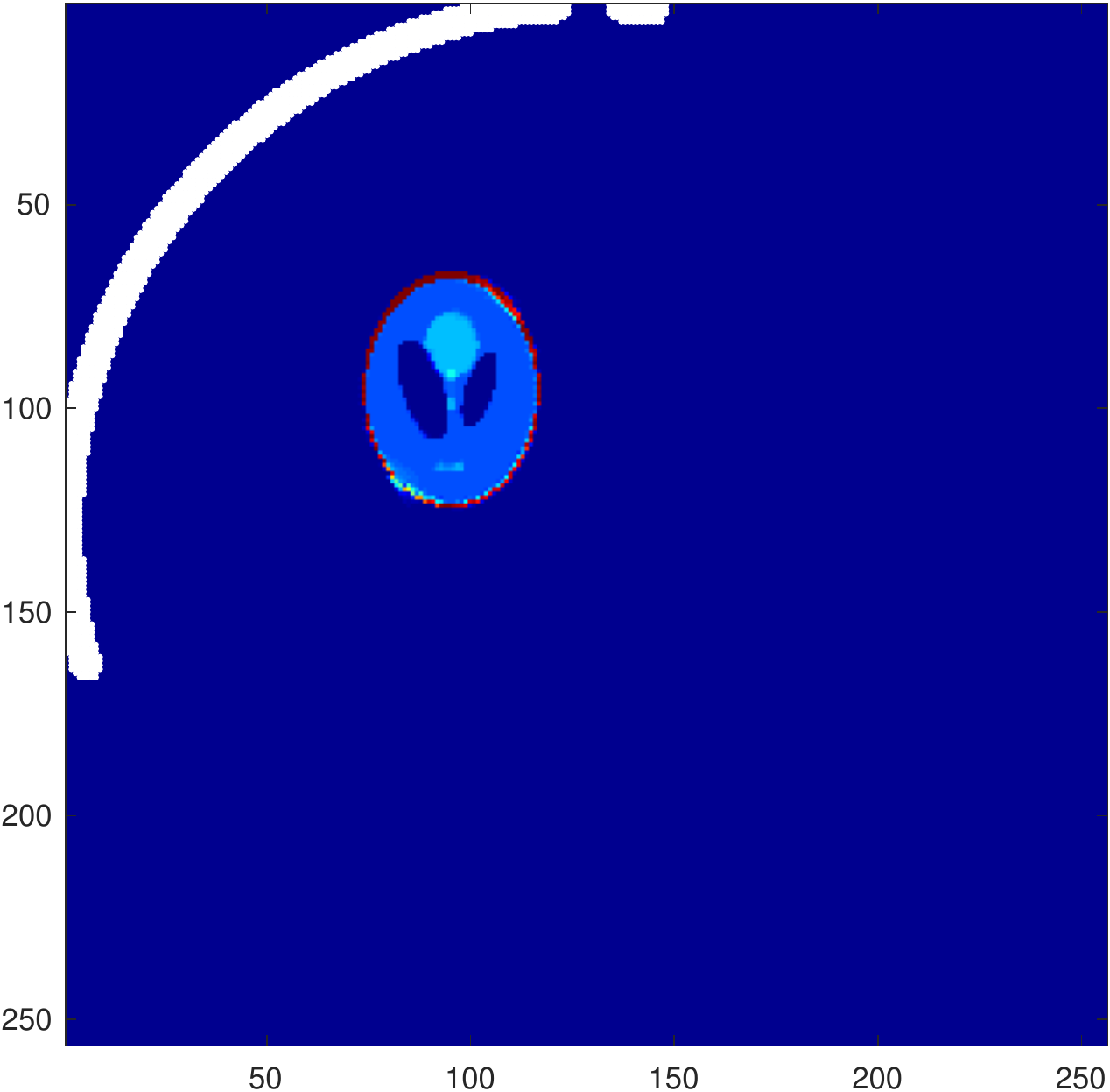} \\
 	\includegraphics[width=0.32\linewidth]{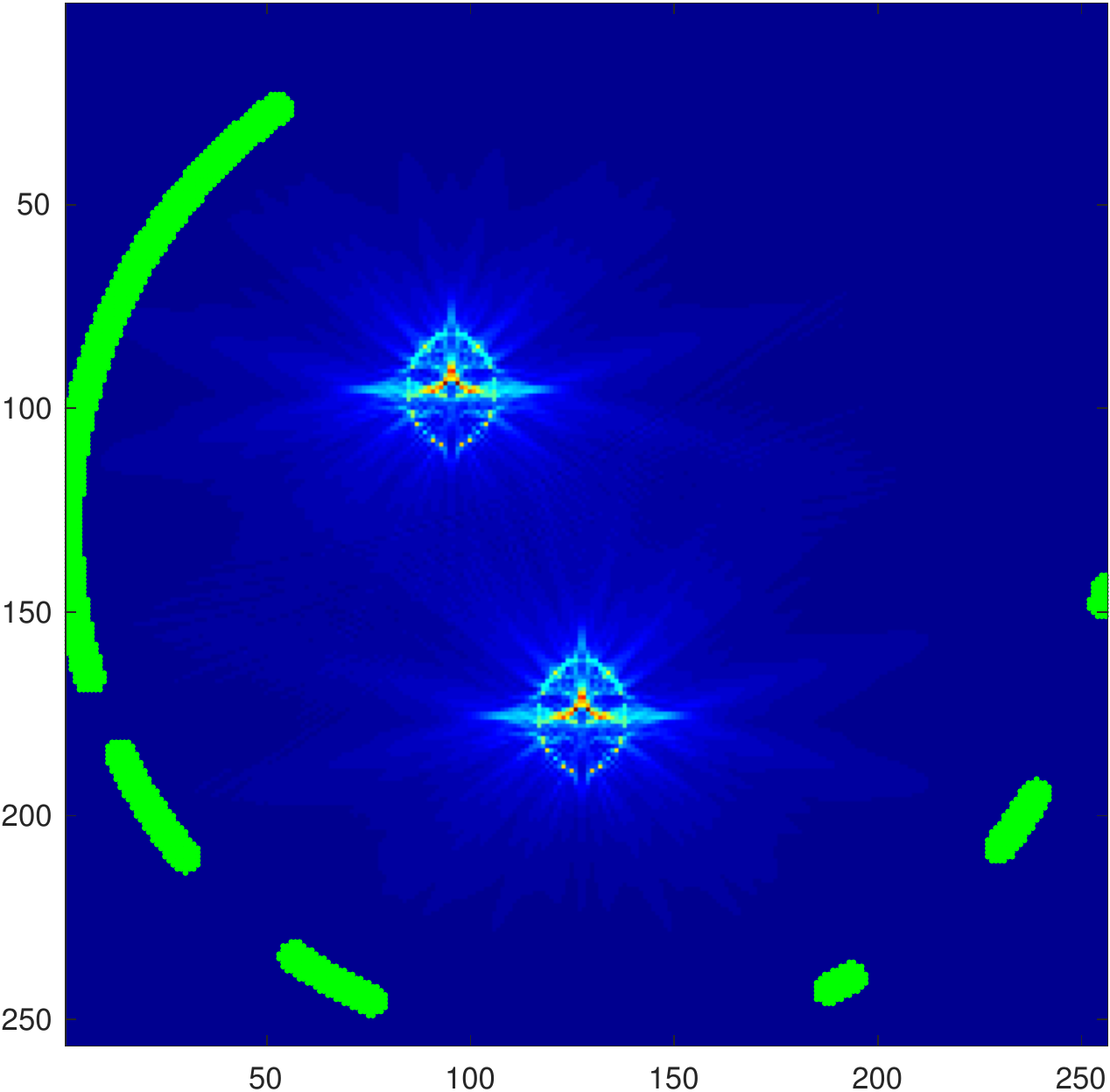}
 	\includegraphics[width=0.32\linewidth]{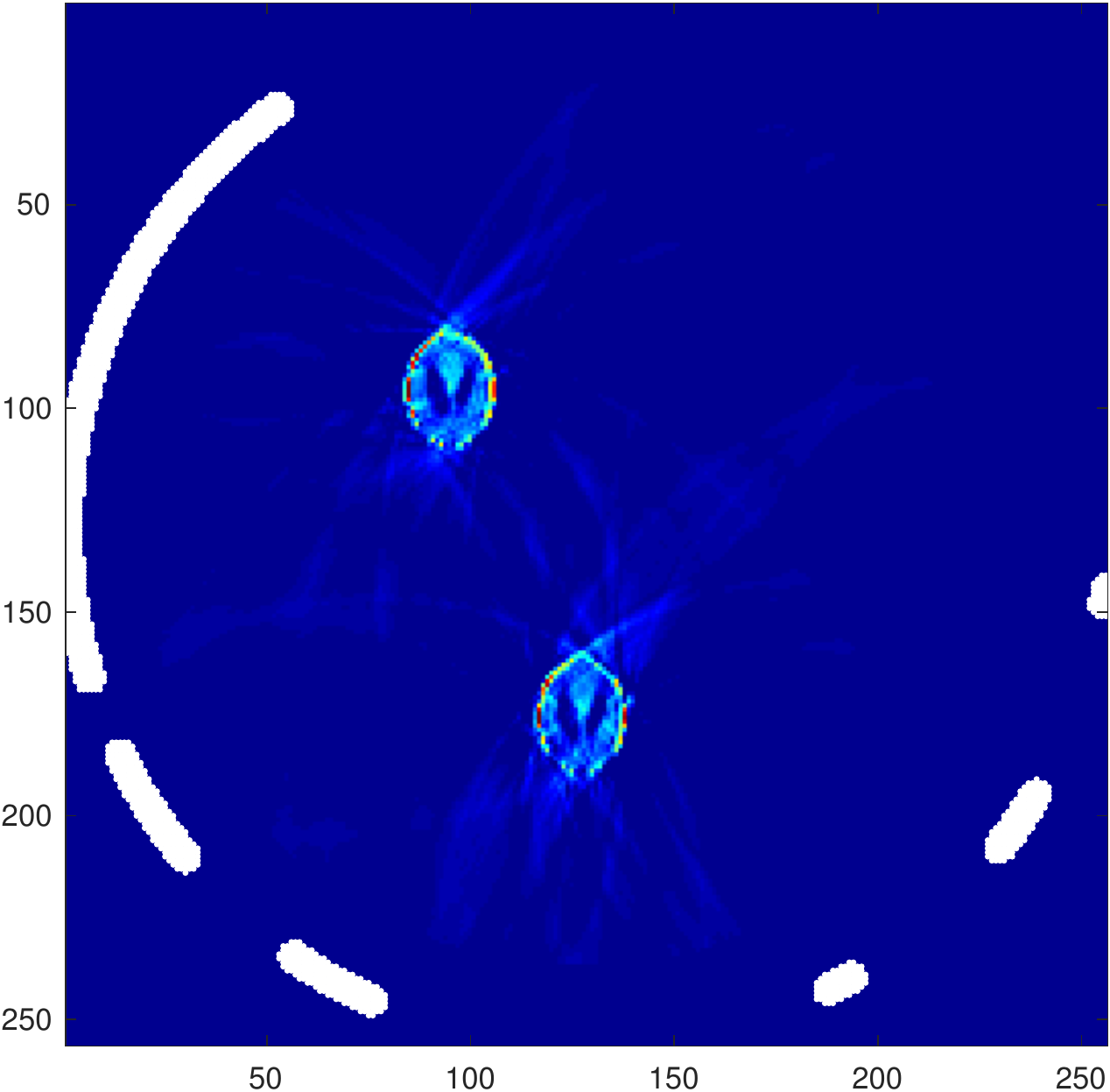} 
 	\includegraphics[width=0.32\linewidth]{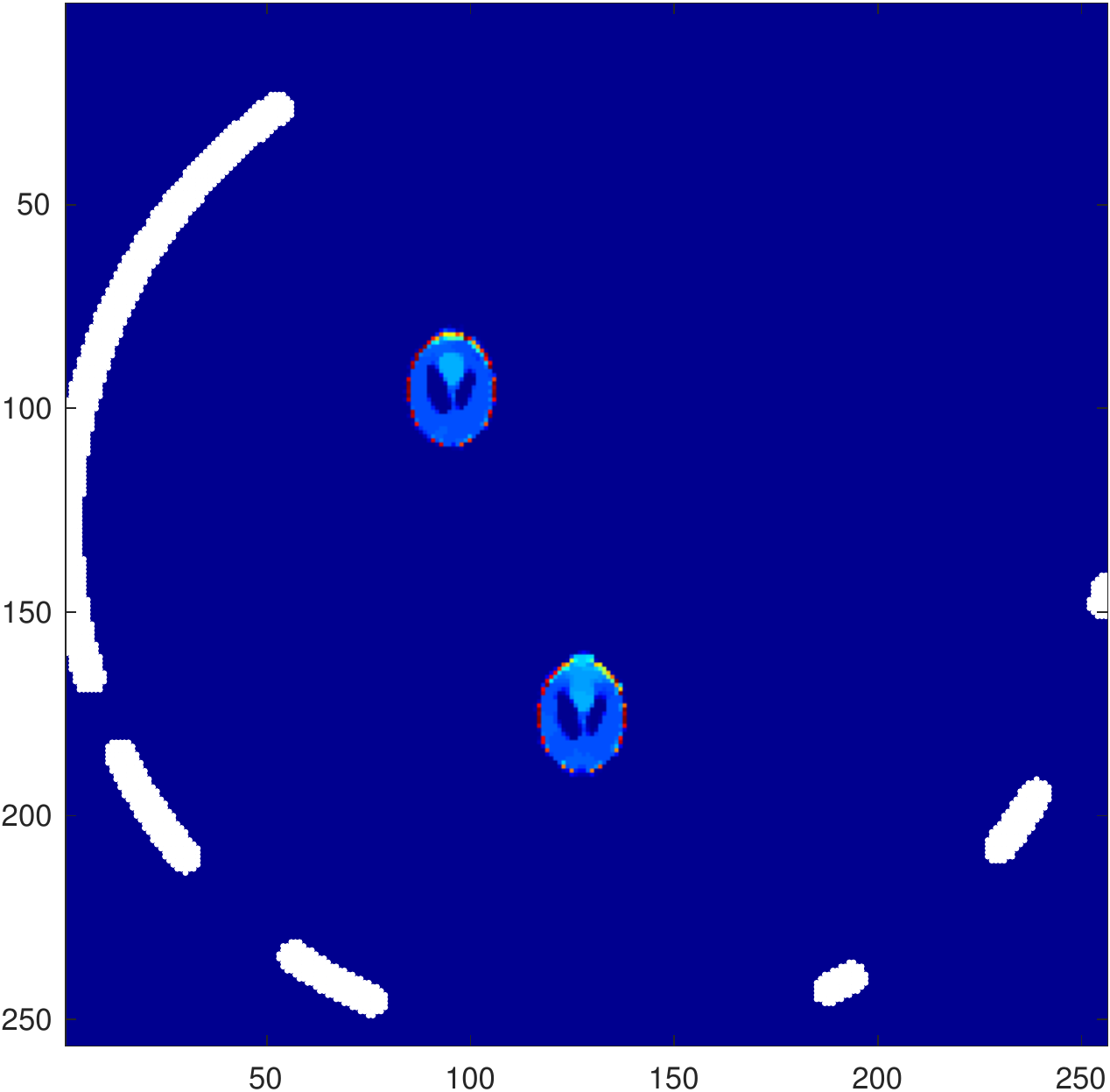} \\
\end{center}
 \caption{Optimization of sensors location. Each line corresponds to a different choice of the source $p_0$. 
 Left: function $ \psi_{[p_0]}(x)$ defined by \eqref{n2202} on $\Omega$ and the best location of sensors (green dots);
 middle: reconstruction by using the time reversal imaging $\I[g]$; 
 right: reconstructed source $p^{n}_{0}$ after $n= 30$ iterations by using the resulting optimal location of sensors.
 \label{fig:time_reversal_iterationTV2}
}
 \end{figure}

 \begin{figure}[h!] 
 \begin{center}
 	\includegraphics[width=0.32\linewidth]{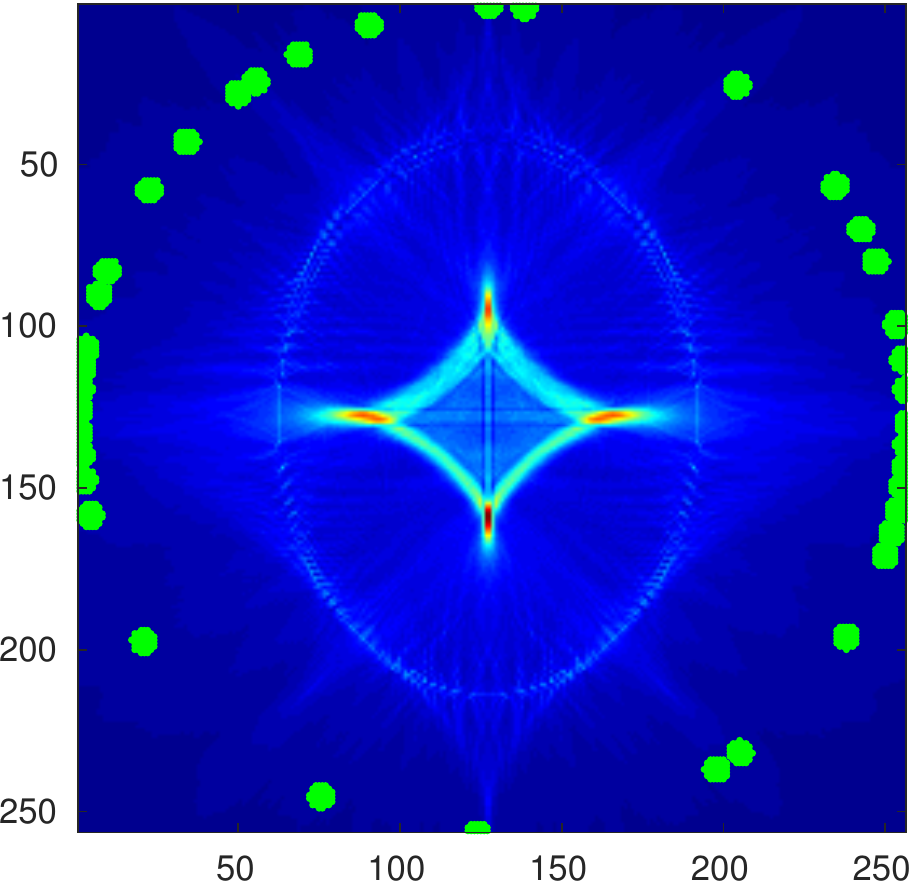}
 	\includegraphics[width=0.32\linewidth]{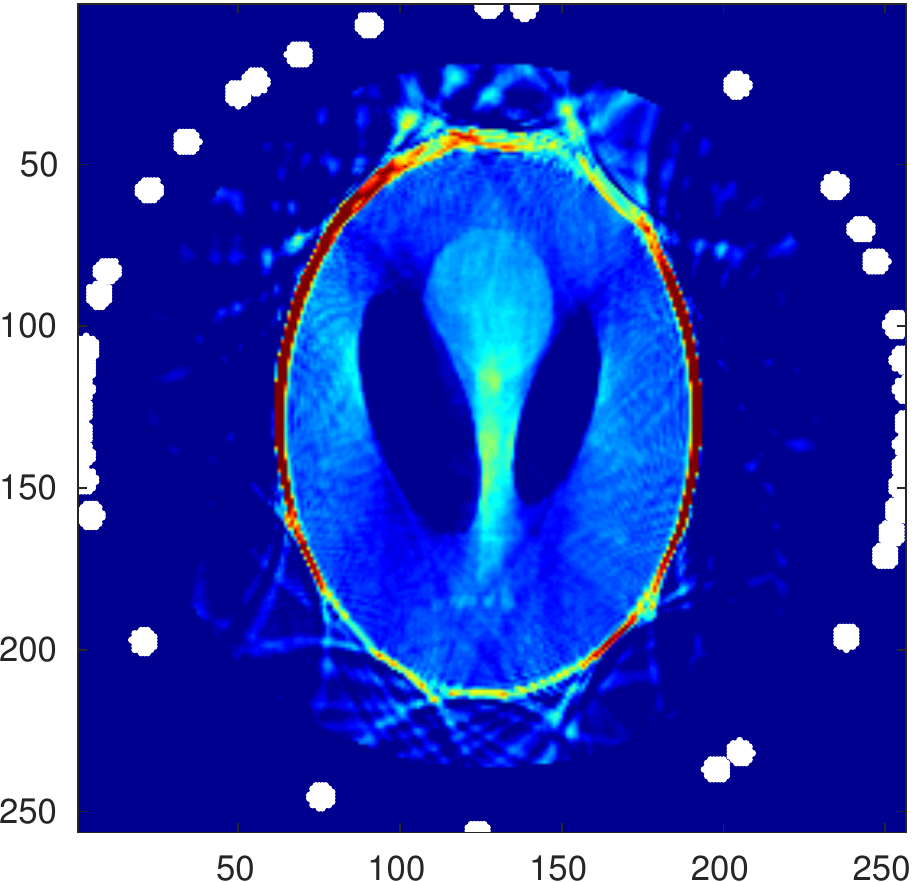} 
 	\includegraphics[width=0.32\linewidth]{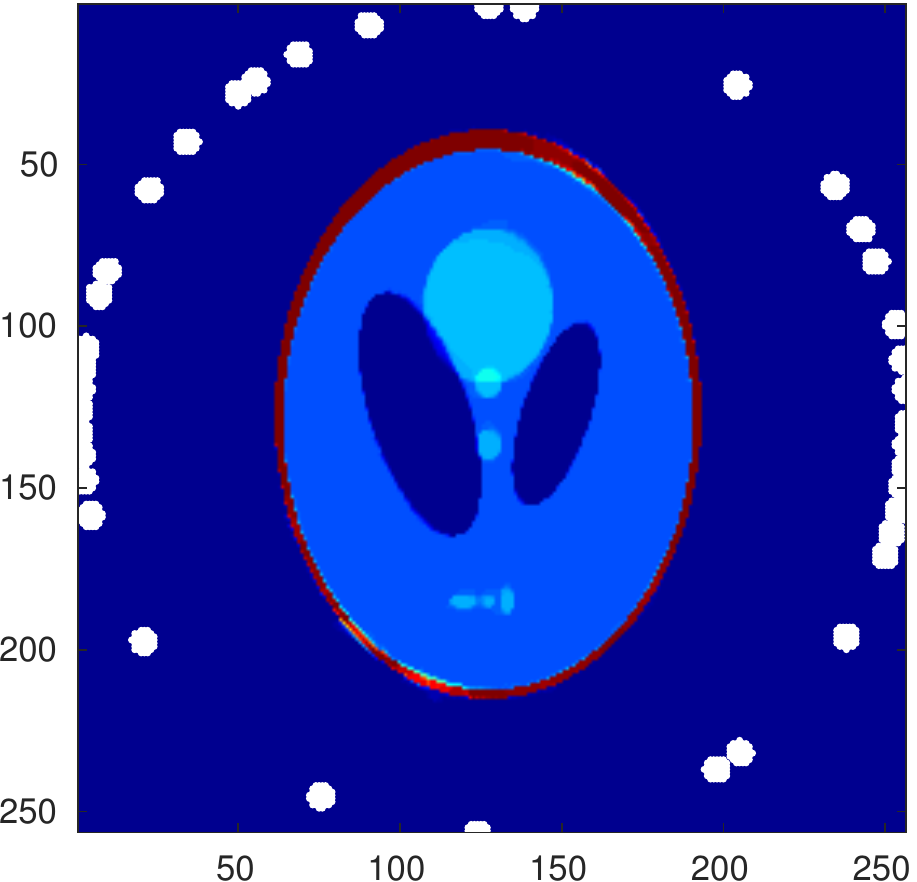} \\
 \includegraphics[width=0.32\linewidth]{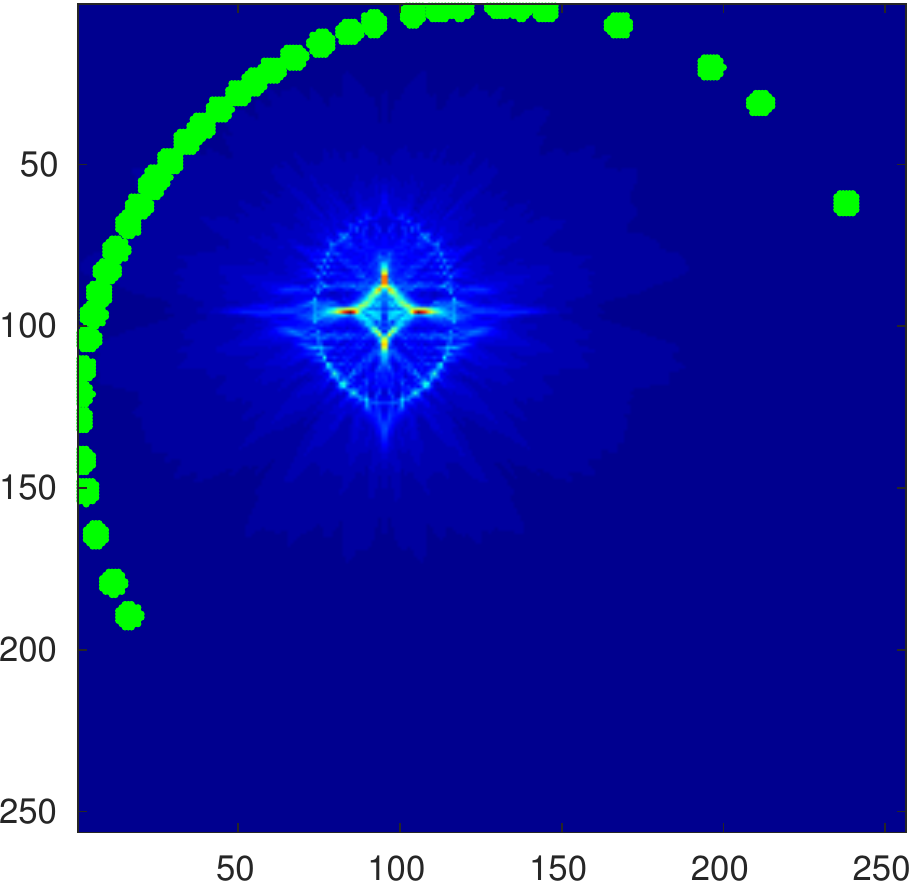}
 	\includegraphics[width=0.32\linewidth]{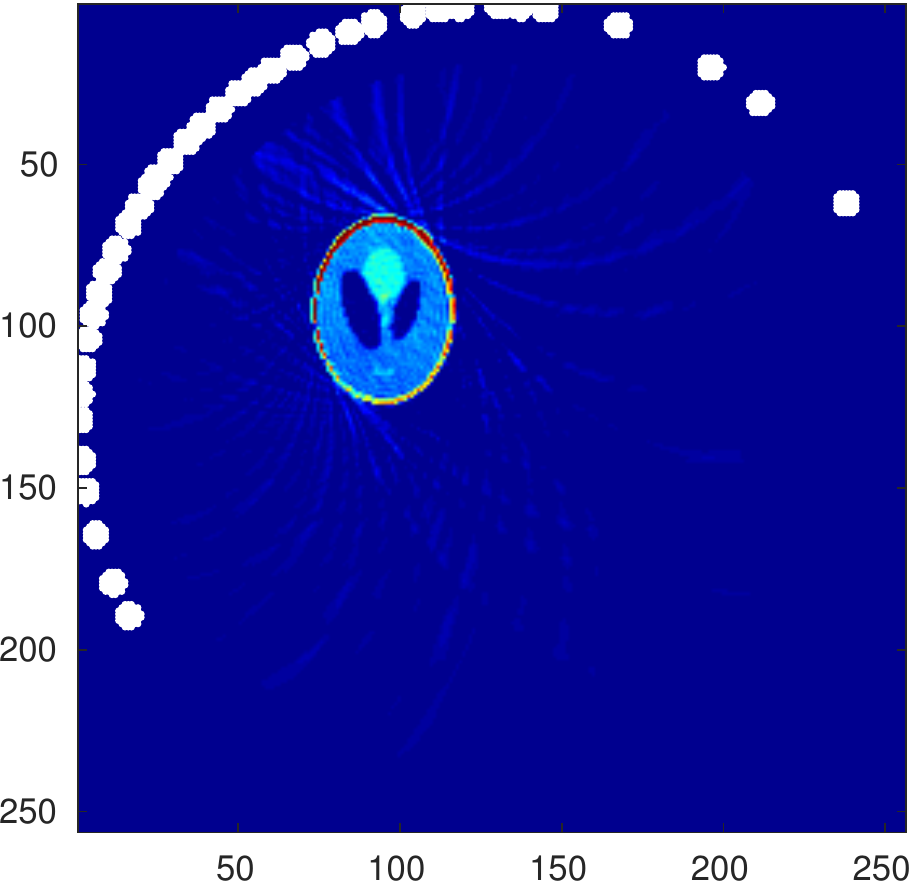} 
 	\includegraphics[width=0.32\linewidth]{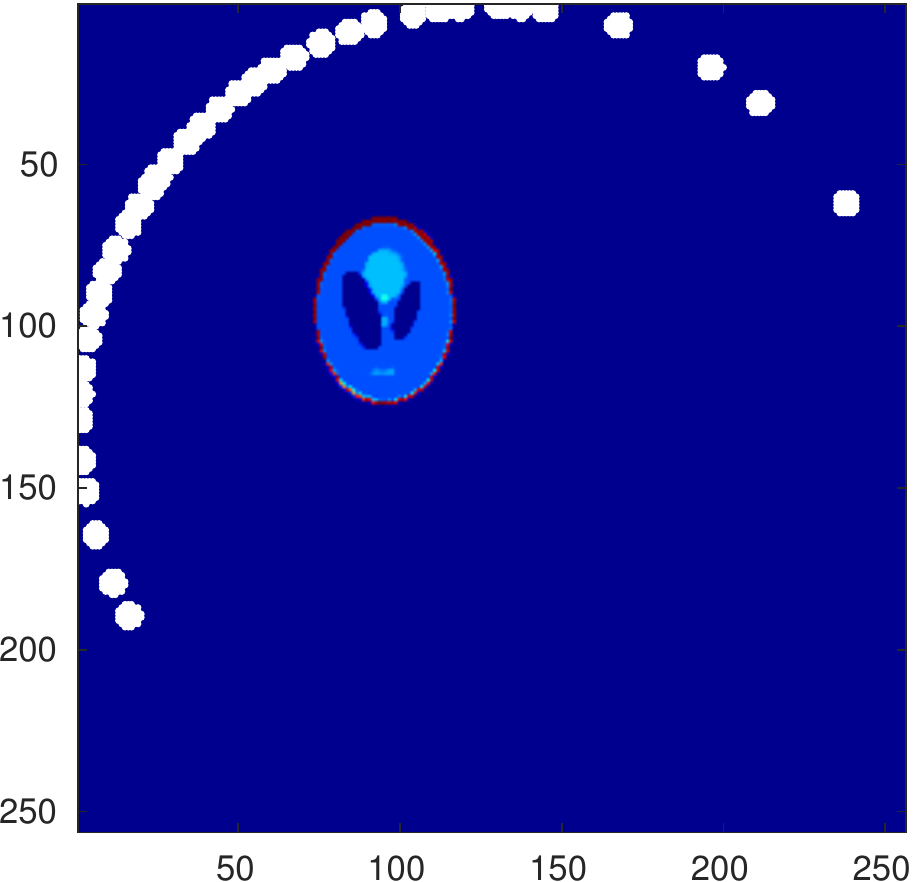} \\
 	\includegraphics[width=0.32\linewidth]{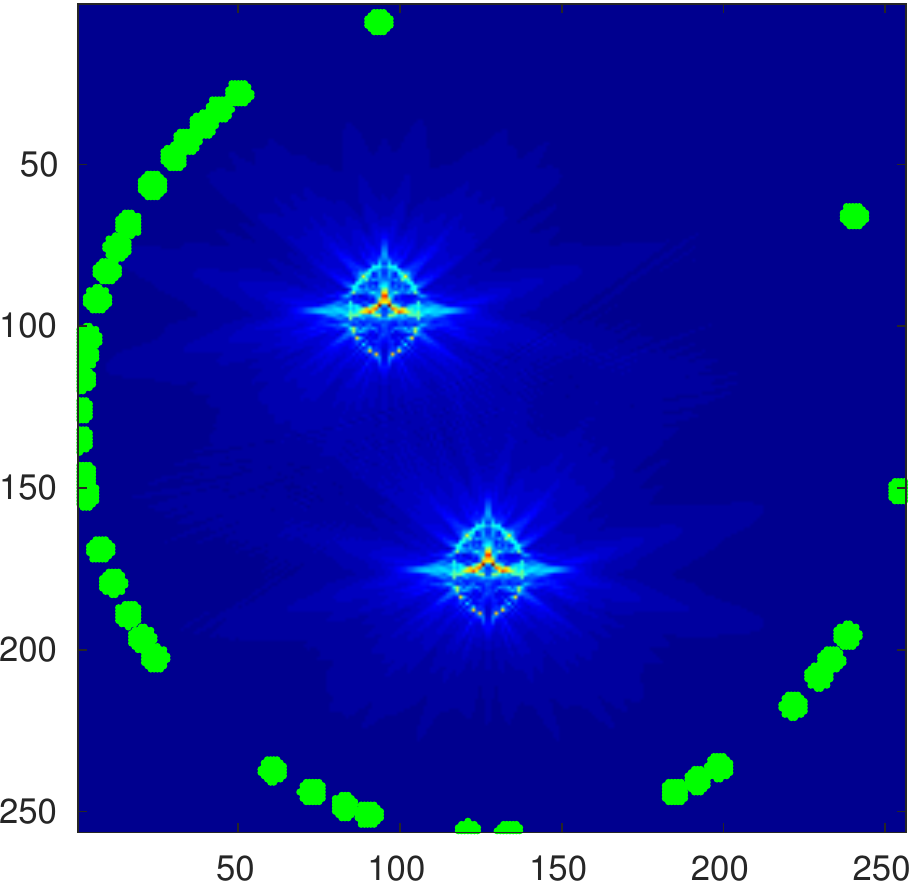}
 	\includegraphics[width=0.32\linewidth]{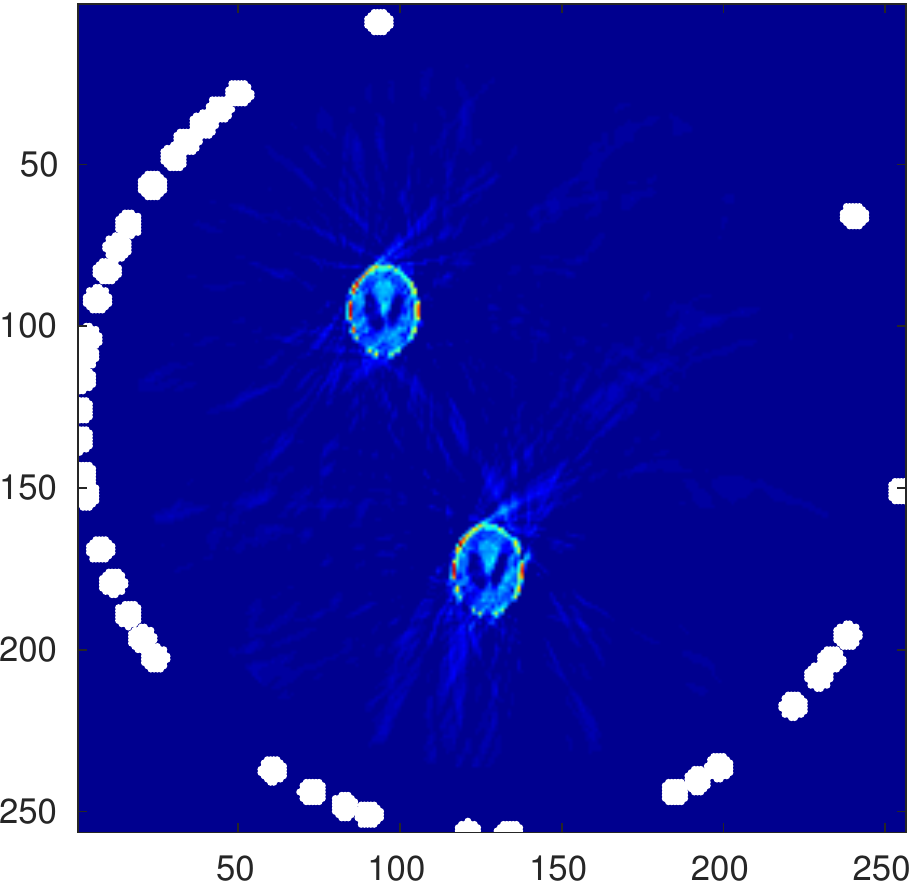} 
 	\includegraphics[width=0.32\linewidth]{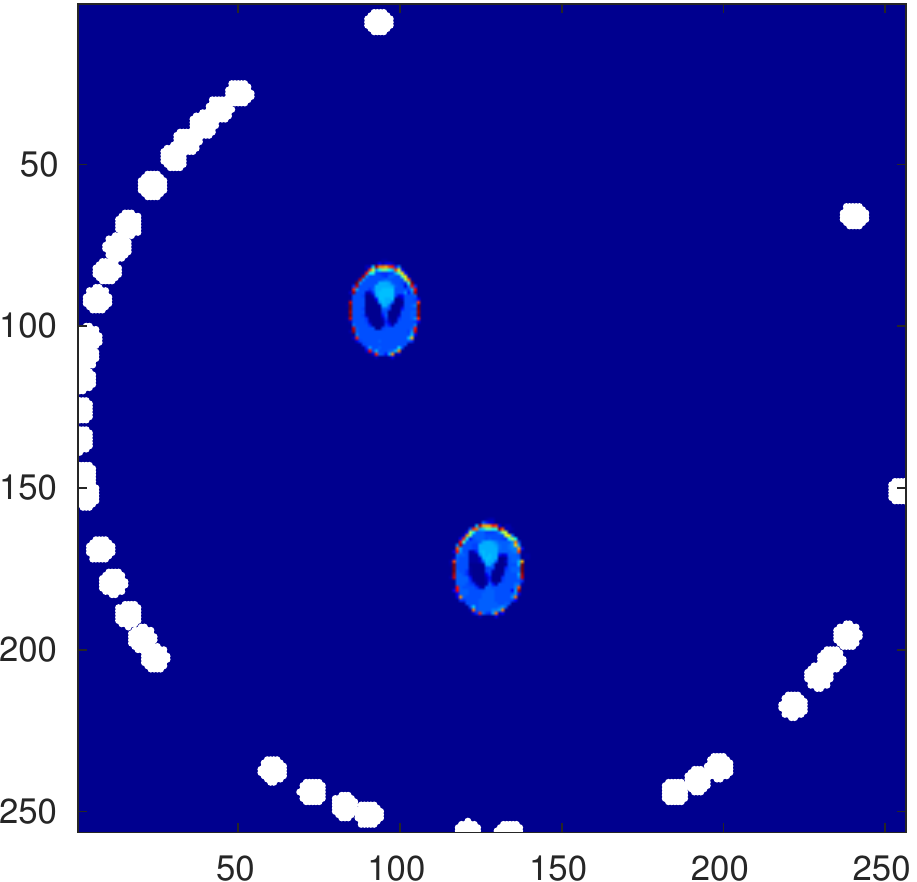} \\
\end{center}
 \caption{Optimization of sensors location with only $38$ sensors.  Each line corresponds to a different choice of the source $p_0$. 
 Left: function $ \psi_{[p_0]}(x)$ defined by \eqref{n2202} on $\Omega$ and the best location of sensors (green dots);
 middle: reconstruction by using the time reversal imaging $\I[g]$; 
 right: reconstructed source $p^{n}_{0}$ after $n= 30$ iterations by using the resulting optimal location of sensors.
 \label{fig:time_reversal_iterationTV3}
}
 \end{figure}

\section{Conclusion}
This article is a first attempt to efficiently locate sensors in the quite sensitive framework of thermo-acoustic tomography. Although the first numerical results seem promising, we foresee to study this issue further by investigating
\begin{itemize} 
\item other choices of modeling. In particular, does it exist better choices of functional $A_2$ as a reconstruction quality factor (for instance other observers)?
\item how to choose more adequately the regularizing term in the optimization problem and investigate the sensitivity of solution with respect to this term?
\item does the iterative scheme consisting of successively estimating the source ${p}_0$ and then a new location of sensors, will converge? In that case, can the limit be identified?
\item can we analyze the relationships between the reconstruction parameters and the number of sensors used for the experiment?
\item how to improve the optimization procedure, especially when dealing with a finite number of sensors where we used a efficient but costly genetic algorithm?
\end{itemize}

More generally, we also plan to collaborate with physicists researchers in view of making experiments and testing our approach and method on true medical imaging data. We believe that the techniques developed within this article can be adapted to many other situations. Nevertheless, it is likely that several additional constraints on the sensor set should be taken into account, typically a restriction on the zones of sensor location.

\bibliographystyle{abbrv}
\bibliography{biblio}

\end{document}